\documentclass[a4paper,12pt]{scrartcl}

\usepackage[pdftex]{graphicx}
\usepackage{amssymb}
\usepackage{amsmath}
\usepackage{subfig}
\usepackage{enumerate}
\usepackage[section]{placeins}

\usepackage[amsmath,thmmarks]{ntheorem}

\theoremstyle{plain}
\newtheorem{theorem}{Theorem}[section]
\newtheorem{corollary}[theorem]{Corollary}

\newtheorem{proposition}[theorem]{Proposition}

\theoremstyle{remark}
\newtheorem{remark}[theorem]{Remark}

\theoremstyle{definition}
\newtheorem{definition}[theorem]{Definition}

\theoremstyle{nonumberplain}
\theoremheaderfont{\normalfont\itshape}
\theorembodyfont{\normalfont}
\theoremsymbol{\ensuremath{\square}}
\theoremseparator{.}
\newtheorem{proof}{Proof}

\numberwithin{equation}{section} 
\numberwithin{figure}{section}
\numberwithin{table}{section}

\usepackage{algpseudocode}
\usepackage{algorithm}
\usepackage{bbm}
\usepackage{xcolor}
\definecolor{graytext}{RGB}{150, 150, 150}
\newcommand{\graytext}[1]{\textcolor{graytext}{#1}}
\makeatletter
\algnewcommand{\LineComment}[1]{\Statex \hskip\ALG@thistlm \graytext{\texttt{\# #1}}}
\algnewcommand{\LineCommentNewEnv}[1]{\Statex \hskip\ALG@thistlm \hskip\algorithmicindent \graytext{\texttt{\# #1}}}
\makeatother

\newcommand{\norm}[1]{\lVert#1\rVert}
\newcommand{\abs}[1]{\lvert#1\rvert} 
\newcommand{\inner}[1]{\langle#1\rangle}

\newcommand{\essinf}{\mathop{\textup{ess\,inf}}}
\newcommand{\esssup}{\mathop{\textup{ess\,sup}}}
\newcommand{\ident}{\mathop{\textup{Id}}}
\newcommand{\sym}{\mathop{\textup{Sym}}}

\newcommand{\rum}[1]{\mathbb{#1}}

\newcommand{\betaL}{\beta^{\textup{L}}}
\newcommand{\betaU}{\beta^{\textup{U}}}
\newcommand{\Dp}{D_{+}}
\newcommand{\Dm}{D_{-}}
\newcommand{\Dpm}{D_{\pm}}
\newcommand{\kp}{\kappa_{+}}
\newcommand{\km}{\kappa_{-}}
\newcommand{\kpm}{\kappa_{\pm}}
\newcommand{\M}{\mathcal{M}}
\newcommand{\Mp}{\M_{+}}
\newcommand{\Mm}{\M_{-}}
\newcommand{\Mpm}{\M_{\pm}}
\newcommand{\Tp}{T_{+}}
\newcommand{\Tm}{T_{-}}
\newcommand{\Tpm}{T_{\pm}}
\newcommand{\Tph}{\Tp^h}
\newcommand{\Tmh}{\Tm^h}
\newcommand{\Tpmh}{\Tpm^h}
\newcommand{\Tphd}{\Tp^{h,\delta}}
\newcommand{\Tmhd}{\Tm^{h,\delta}}
\newcommand{\Tpmhd}{\Tpm^{h,\delta}}
\newcommand{\ap}{\alpha_{+}}
\newcommand{\am}{\alpha_{-}}
\newcommand{\apm}{\alpha_{\pm}}
\newcommand{\apL}{\ap^{\textup{L}}}
\newcommand{\apU}{\ap^{\textup{U}}}
\newcommand{\amL}{\am^{\textup{L}}}
\newcommand{\amU}{\am^{\textup{U}}}
\newcommand{\apmL}{\apm^\textup{L}}
\newcommand{\apmU}{\apm^\textup{U}}

\newcommand{\XXint}[3]{{\setbox0=\hbox{$#1{#2#3}{\int}$}
		\vcenter{\hbox{$#2#3$}}\kern-.5\wd0}}

\makeatletter
\def\blfootnote{\gdef\@thefnmark{}\@footnotetext}
\makeatother

\title{\LARGE{The regularized monotonicity method:\\ detecting irregular indefinite inclusions\blfootnote{This research is funded by grant 4002--00123 \emph{Improved Impedance Tomography with Hybrid Data} from The Danish Council for Independent Research \textbar\ Natural Sciences.}}}

\author{}
\date{}

\begin{document}
	
\maketitle

\vspace{-2.5cm}

\centerline{\scshape Henrik Garde}
\medskip
{\footnotesize
	\centerline{Department of Mathematical Sciences, Aalborg University}
	\centerline{Skjernvej 4A, 9220 Aalborg, Denmark}
} 

\medskip

\centerline{\scshape Stratos Staboulis}
\medskip
{\footnotesize
	\centerline{Eniram Oy (A W\"artsil\"a company)}
	\centerline{It\"alahdenkatu 22a, 00210 Helsinki, Finland}
}


\begin{abstract}
	\footnotesize
	\textbf{Abstract}. In inclusion detection in electrical impedance tomography, the support of perturbations (inclusion) from a known background conductivity is typically reconstructed from idealized continuum data modelled by a Neumann-to-Dirichlet map. Only few reconstruction methods apply when detecting indefinite inclusions, where the conductivity distribution has both more and less conductive parts relative to the background conductivity; one such method is the monotonicity method of Harrach, Seo, and Ullrich \cite{Harrach13,Harrach10}. We formulate the method for irregular indefinite inclusions, meaning that we make no regularity assumptions on the conductivity perturbations nor on the inclusion boundaries. We show, provided that the perturbations are bounded away from zero, that the outer support of the positive and negative parts of the inclusions can be reconstructed independently. Moreover, we formulate a regularization scheme that applies to a class of approximative measurement models, including the Complete Electrode Model, hence making the method robust against modelling error and noise. In particular, we demonstrate that for a convergent family of approximative models there exists a sequence of regularization parameters such that the outer shape of the inclusions is asymptotically exactly characterized. Finally, a peeling-type reconstruction algorithm is presented and, for the first time in literature, numerical examples of monotonicity reconstructions for indefinite inclusions are presented.
	
	\vspace{.5cm}
	
	\vspace{.5cm}
\end{abstract} 

\section{Introduction}

In \emph{electrical impedance tomography} (EIT), internal information about the electrical conductivity distribution of a physical object is reconstructed from current and voltage measurements taken at surface electrodes. More precisely, by prescribing currents in a basis of current patterns and measuring the corresponding voltage, the current-to-voltage map can be obtained. By partly solving the ill-conditioned {\em Inverse Conductivity Problem}, information about the conductivity, such as the locations and shapes of inclusions in a known background, can be determined. Examples of  applications include monitoring patient lung function, control of industrial processes, non-destructive testing of materials, and locating mineral deposits \cite{Borcea2002a,Cheney1999,Hanke2003,Uhlmann2009,Holder2005,Abubakar2009,York2001,Karhunen2010,Karhunen2010a}. 

The governing equation for mathematical models of EIT is the \emph{conductivity equation}
\begin{equation}\label{eq:condeq}
\nabla\cdot (\gamma\nabla u) = 0, \enskip \text{in } \Omega,
\end{equation}
where $\Omega\subset \mathbb{R}^d$ describes the spatial dimensions of the object, $\gamma$ is the conductivity distribution, and $u$ is the electric potential. The conductivity equation \eqref{eq:condeq} describes the physical interplay of the conductivity and the electric potential under a static electric field. In the \emph{continuum model} (CM), the current-to-voltage map is modelled by a \emph{Neumann-to-Dirichlet} (ND) map $\Lambda(\gamma)$ which relates any applied current density at the boundary to the boundary potential $u|_{\partial\Omega}$ determined through \eqref{eq:condeq} and a Neumann condition. In general, unique and stable determination of inclusions has only been rigorously proven for the CM whereas unique solvability results are typically out of reach when using realistic finite dimensional electrode models. Nevertheless, it is possible to obtain reasonable approximations to the inclusions using practically relevant electrode models \cite{GardeStaboulis_2016,Harrach15,Lechleiter2008a}; to get meaningful reconstructions it is essential to associate the approximate solution to the ideal reconstruction determined by the CM, for instance through regularization theory. 

So far, EIT algorithms for inclusion detection have primarily been implemented for the reconstruction of \emph{definite} inclusions, meaning that the conductivity has either only positive or only negative perturbations to the background conductivity. Reconstruction methods for definite inclusions mainly comprise, for instance, the {\em monotonicity method} \cite{Harrach13,Tamburrino2002,Tamburrino2006,Harrach10,GardeStaboulis_2016,Harrach15,Garde_2017a,Harrach2016}, the {\em factorization method} \cite{Bruhl2001,Bruhl2000,Kirsch2008,Hanke2015,Harrach13b,Lechleiter2008a,Lechleiter2006}, and the {\em enclosure method} \cite{Ikehata1999a,Ikehata2000c,Ikehata2002a,Brander_2015}. In this paper we focus on \emph{indefinite} inclusions which consist of both positive and negative perturbations relative to the background conductivity. Only a few theoretical identifiability results exist for the indefinite case. For example, in the factorization method \cite{Schmitt_2009,Grinberg_2004} it is required that the domain can be partitioned into two components: one containing the positive inclusions and another the negative ones. The most general result concerns the monotonicity method. It has been shown that there is no need to a priori partition the domain into positive and negative components \cite{Harrach13}. 

The main idea behind the monotonicity method can be illustrated by a simple example. Suppose the conductivity is of the form $\gamma = 1+\chi_{\Dp} - \tfrac{1}{2}\chi_{\Dm}$, where $\chi_{\Dp}$ and $\chi_{\Dm}$ are characteristic functions on the open and disjoint sets $\Dp,\Dm\subseteq\Omega$ that constitute the unknown indefinite inclusion $D=\Dp\cup\Dm$ that we wish to reconstruct. For any measurable $C\subseteq \overline{\Omega}$ we have by monotonicity that $D\subseteq C$ implies $\Lambda(1+\chi_C)\leq \Lambda(\gamma)\leq \Lambda(1-\tfrac{1}{2}\chi_C)$. Recently, it was shown that the converse holds if $C$ is closed and has connected complement \cite{Harrach13}: If $\M$ is the collection of all such sets $C$ satisfying the above inequalities, then $\cap\M$ characterizes the outer support of $D$ by coinciding with the smallest closed set containing $D$ and having connected complement. In particular, if $\overline{D}$ has no holes, then the reconstruction coincides with $\overline{D}$. Remarkably, the result remains valid when replacing the operator inequality with the affine approximation $\Lambda'(1)\chi_C \leq \Lambda(\gamma)-\Lambda(1)\leq -\Lambda'(1)\chi_C$. In practice the affine formulation transforms into fast numerical implementations that mostly rely on cheap matrix-vector products. 

In the recent paper \cite{GardeStaboulis_2016} a regularized version of the monotonicity method for definite inclusions was studied and asymptotically tied with electrode modelling. In this paper's main result Theorem \ref{thm:conv} we extend the theory from definite inclusions to indefinite inclusions. We show that a sequence of discrete models that approximate the CM controllably provides a sequence of monotonicity reconstructions that converge to the CM counterpart as the measurement noise and modelling error tend to zero. In particular, it can be shown that favourable approximative models include the {\em Complete Electrode Model} (CEM) \cite{Somersalo1992,Hyvonen2017} which takes the shunting effect and imperfect electrode contacts into account.

For the CM we extend the results in \cite{Harrach13} by showing that non-smooth $L^\infty$-perturbations are allowed and that they satisfy the required \emph{definiteness condition}, as long as the perturbations are essentially bounded away from the underlying background conductivity. Since we consider $L^\infty$-perturbations and do not take any regularity assumptions for the boundaries of the inclusions, we call the inclusions \emph{irregular}. Furthermore, we prove in Theorem \ref{thm:indep_mono} that if the outer boundary of the positive and negative parts of the inclusions can be connected to the domain boundary by only traversing the background conductivity, then the positive and negative inclusions can be reconstructed independently of one another. A reconstruction algorithm is presented and, for the first time, examples of numerical reconstruction are given for the monotonicity method for indefinite inclusions.

To simplify the presentation of the results and proofs, we restrict our attention to the case where $\overline{\Omega}$ has connected complement and use full boundary Cauchy data for the CM. However, we do note that more generally one can consider the monotonicity method for domains with holes, and in addition only have Cauchy data on a subset of the boundary, similar to the considerations in \cite[Section 4.3]{Harrach13}. Furthermore, the proof of our main result in Theorem~\ref{thm:conv} does not directly depend on whether we consider local Cauchy data, as long as the approximate operators satisfy an estimate corresponding to \eqref{lamhbnd}.

The paper is organized as follows. In Section~\ref{sec:cm} we introduce the CM and give the essential monotonicity properties of the ND map and its Fr\'echet derivative. The monotonicity method for the CM is outlined in Section~\ref{sec:mm}. In particular, Theorem~\ref{thm:mono_continuum} gives a simple proof of the method for irregular indefinite inclusions thus forming the main framework of the paper. The proof also enables expressing the conditions under which the positive and negative part of the inclusions can be reconstructed separately; the related results are stated and proven in Theorem~\ref{thm:indep_mono}. In Section~\ref{sec:rmm} the regularized monotonicity method is formulated and the main result Theorem~\ref{thm:conv} of the paper is proven. Finally, a peeling-type algorithm is constructed in Section~\ref{sec:numexp} for implementing the regularized monotonicity method, and several numerical examples employing the CEM are presented.

\subsection{Notational remarks}

For brevity we denote the essential infima/suprema $\essinf_{x\in\Omega} f(x)$ and $\esssup_{x\in\Omega} f(x)$ by $\inf(f)$ and $\sup(f)$, respectively. Since the indefinite inclusions and the operators involved in the monotonicity method require far more notation compared to the definite case, we constantly employ the symbols $"+"$/$"-"$ to associate sets and operators to positive/negative inclusions. To avoid excessive repetition, we often use the notation $\pm$ to indicate that a set inclusion or equation/inequality holds for both the $"+"$ and $"-"$ version of a set or operator. For example, $\Mpm \subseteq \M_{\apm}(\Tpmhd)$ states that both set inclusions $\Mp \subseteq \M_{\ap}(\Tphd)$ and $\Mm \subseteq \M_{\am}(\Tmhd)$ hold true. 

\section{Reconstruction of indefinite inclusions based on monotonicity}

In this section the CM is formulated for an isotropic conductivity distribution. The essential monotonicity properties of the associated Neumann-to-Dirichlet map are revised to motivate the monotonicity method. Furthermore, in Section \ref{sec:mm} the monotonicity method is formulated for the CM for reconstruction of irregular indefinite inclusions.

\subsection{The continuum model} \label{sec:cm}

Let $\Omega\subset \mathbb{R}^d$ for $d\geq 2$ be an open and bounded domain with $C^\infty$-regular boundary $\partial\Omega$. To simplify some arguments, we will also assume that $\rum{R}^d\setminus\overline{\Omega}$ is connected. 

The CM is governed by the following elliptic boundary value problem
\begin{equation}
\nabla\cdot(\gamma\nabla u) = 0 \text{ in } \Omega, \quad \nu\cdot \gamma\nabla u = f \text{ on } \partial\Omega, \quad \int_{\partial\Omega} u\,dS = 0, \label{eq:cm}
\end{equation}
where $\nu$ is an outwards-pointing unit normal to $\partial\Omega$ and the real-valued conductivity distribution $\gamma$ belongs to
\begin{equation}
L^\infty_+(\Omega) = \{ w\in L^\infty(\Omega) : \inf(w) > 0 \}.
\end{equation}
The latter condition in \eqref{eq:cm} corresponds to a grounding of the electric potential. The Neumann boundary condition models the current density applied to the object at its boundary. According to standard elliptic theory, for any 
\begin{equation*}
f\in L^2_\diamond(\partial\Omega) = \{ w\in L^2(\partial \Omega) : \inner{w,1} = 0 \},
\end{equation*}
problem \eqref{eq:cm} has a unique weak solution $u$ in
\begin{equation*}
H_\diamond^1(\Omega) = \{ w\in H^1(\Omega) : \inner{w|_{\partial\Omega},\mathbbm{1}} = 0 \}.
\end{equation*}
Here and in the remainder of the paper, $\inner{\cdot,\cdot}$ denotes the usual  $L^2(\partial\Omega)$-inner product, $w|_{\partial\Omega}$ denotes the trace of $w$ on $\partial\Omega$, and $\mathbbm{1} \equiv 1$ on $\partial\Omega$. This gives rise to a well-defined ND map
\begin{equation*}
\Lambda(\gamma) : L^2_\diamond(\partial\Omega)\to L^2_\diamond(\partial\Omega),\quad f\mapsto u|_{\partial\Omega}.
\end{equation*}
The graph of the ND map corresponds to all possible pairs of applied current densities and measured boundary voltage, and as such $\Lambda(\gamma)$ models the (infinite precision) current-to-voltage datum for the CM. Note that $\Lambda(\gamma)$ belongs to $\mathcal{L}(L^2_\diamond(\partial\Omega))$, the space of linear and bounded operators on $L^2_\diamond(\partial\Omega)$, for each $\gamma \in L^\infty_+(\Omega)$. Moreover, $\Lambda(\gamma)$ is a compact and self-adjoint operator.

Although the forward map $\Lambda : L^\infty_+(\Omega) \to \mathcal{L}(L^2_\diamond(\partial\Omega))$ is non-linear, it is Fr\'echet differentiable on $L^\infty_+(\Omega)$ and the derivative is determined through the quadratic form (cf.\ \cite[Appendix B]{GardeStaboulis_2016})
\begin{equation}
\inner{\Lambda'(\gamma)[\eta]f,f} = -\int_{\Omega} \eta \abs{\nabla u}^2\,dx, \enskip \eta\in L^\infty(\Omega) \label{eq:frechetdiff}
\end{equation}
where $u$ is the solution to \eqref{eq:cm} for a conductivity distribution $\gamma$ and Neumann condition $f$. The following proposition, along with \eqref{eq:frechetdiff}, represents the essential monotonicity principles of $\Lambda$ and $\Lambda'$ which form the foundation of the monotonicity method.
\begin{proposition} \label{prop:monorel}
	For $f\in L^2_\diamond(\partial\Omega)$ and $\gamma,\tilde{\gamma}\in L^\infty_+(\Omega)$ it holds
	\begin{equation}
	\int_{\Omega} \frac{\tilde{\gamma}}{\gamma}(\gamma-\tilde{\gamma})\abs{\nabla \tilde{u}}^2\,dx \leq \inner{(\Lambda(\tilde{\gamma})-\Lambda(\gamma))f,f} \leq \int_{\Omega} (\gamma-\tilde{\gamma})\abs{\nabla\tilde{u}}^2\,dx
	\end{equation}	
	where $\tilde{u}$ is the solution to \eqref{eq:cm} with conductivity $\tilde{\gamma}$ and Neumann condition $f$.
\end{proposition}
\begin{proof}
	For a recent proof see, e.g., \cite[Lemma 3.1]{Harrach13} or \cite[Lemma 2.1]{Harrach10}. See also \cite{Ikehata1998a,Kang1997b} for similar monotonicity estimates.
\end{proof}

An immediate consequence of Proposition \ref{prop:monorel} is the simple monotonicity relation
\begin{equation}
\gamma \geq \tilde{\gamma} \text{ a.e.\ in }\Omega \quad \text{implies} \quad \Lambda(\tilde{\gamma}) \geq \Lambda(\gamma). \label{eq:simplemono}
\end{equation}
\begin{remark}
	In \eqref{eq:simplemono} and in the remainder of the paper, operator inequalities of the type $A\leq B\leq C$ for self-adjoint $A,B,C\in\mathcal{L}(L^2_\diamond(\partial\Omega))$ are understood in the sense of positive semi-definiteness, that is, $\inner{(C-B)f,f}\geq 0$ and $\inner{(B-A)f,f}\geq 0$ for all $f\in L^2_\diamond(\partial\Omega)$. 
\end{remark}

\subsection{The monotonicity method for irregular indefinite inclusions} \label{sec:mm}

In this section we review the general theory related to the monotonicity method for the CM based reconstruction of indefinite inclusions. Moreover, we introduce the notation that will be used later in Section \ref{sec:rmm}.  

A concept essential for characterizing the reconstructions from the monotonicity method is the \emph{outer support} of a set.  
\begin{definition}[Outer support] \label{def:outsupp}
	The outer support $X^\bullet$ of a subset $X\subseteq \overline{\Omega}$ is defined as
	\begin{equation*}
	X^\bullet = \overline{\Omega}\setminus \cup \{ U\subseteq \mathbb{R}^d\setminus X \text{ open and connected} : U\cap \partial\Omega \neq \emptyset \}.
	\end{equation*}
\end{definition}

We state a less technical characterization of the outer support in the following proposition.

\begin{proposition} \label{prop:outersupp}
	If $\rum{R}^d\setminus \overline{\Omega}$ is connected and $X\subseteq \overline{\Omega}$, then $X^\bullet$ is the smallest closed set containing $X$ and having connected complement in $\rum{R}^d$. As a consequence, if $\rum{R}^d\setminus\overline{X}$ is connected then $X^\bullet = \overline{X}$.
\end{proposition}
\begin{proof}
	By Definition~\ref{def:outsupp}, $X^\bullet$ is closed and $X\subseteq X^\bullet$ which implies $\overline{X}\subseteq X^\bullet$. If $\partial\Omega\subseteq \overline{X}$ then clearly $X^\bullet = \overline{\Omega}$ and we are done.
	
	Assume from now on that $\partial\Omega\not\subseteq \overline{X}$, then
	\begin{equation*}
	W = \cup \{ U\subseteq \mathbb{R}^d\setminus X \text{ open and connected} : U\cap \partial\Omega \neq \emptyset \}
	\end{equation*}
	is non-empty. Since $\rum{R}^d\setminus \overline{\Omega}$ is connected then $\rum{R}^d\setminus \overline{\Omega}\subset W$, and hence $W$ is open and connected. In particular, $X^\bullet = \rum{R}^d\setminus W$ so $X^\bullet$ has connected complement in $\rum{R}^d$.
	
	Let $x\in \partial X^\bullet$, then for any open neighbourhood $V$ of $x$ it holds that $V\cap W\neq \emptyset$ by the definition of $X^\bullet$. Also $V\cap X \neq \emptyset$, since otherwise $x\in W$. Hence $\partial X^\bullet\subseteq \partial X$. 
	
	Now let $U$ denote an open set. If $U\cap \partial X^\bullet\neq \emptyset$ then $X^\bullet\setminus U$ no longer contains $X$, and if $U\cap \partial X^\bullet = \emptyset$ but $U\cap X^\bullet \neq \emptyset$ then $X^\bullet\setminus U$ no longer has connected complement in $\rum{R}^d$. Thus, we conclude that $X^\bullet$ is the smallest closed set containing $X$ and having connected complement in $\rum{R}^d$.
\end{proof}

Throughout this paper we will use the term \emph{connected complement} to refer to the connectedness of the complement in $\rum{R}^d$, and for simplicity we also assume that $\overline{\Omega}$ has connected complement. Hence, by Proposition~\ref{prop:outersupp} we can consistently state that any closed set $C\subseteq \overline{\Omega}$ with connected complement satisfies $C = C^\bullet$. Such sets will appear in most of the results for the monotonicity method. Consequently, we define the family of \emph{admissible test inclusions} as
\begin{equation}
\mathcal{A} = \{ C\subseteq \overline{\Omega} \text{ closed set} : \mathbb{R}^d \setminus C \text{ connected} \}.
\end{equation}
Now we are ready to define the notion of an indefinite inclusion.

\begin{definition}[Indefinite inclusion] \label{def:indefinc}
	Consider a conductivity distribution of the form 
	\begin{equation}
	\gamma = \gamma_0 + \kp\chi_{\Dp} - \km\chi_{\Dm}
	\end{equation}
	where $\gamma_0,\kpm\in L^\infty_+(\Omega)$, $\sup(\km) < \inf(\gamma_0)$, and $\Dpm\subseteq \Omega$ are open and disjoint. The set $D = \Dp\cup \Dm$ is called an \emph{indefinite inclusion} with respect to the background conductivity $\gamma_0$. The sets $\Dp$ and $\Dm$ are called the \emph{positive part} and \emph{negative part}, respectively, of the inclusion. In addition, using the notation $\Dpm^\bullet = (\Dpm)^\bullet$, we also have the following technical assumption: Any open neighbourhood of $x\in\partial D^\bullet$ contains an open ball $B$ satisfying
	\begin{enumerate}[(i)]
		\item $B\cap \partial D^\bullet \neq \emptyset$,
		\item Either $B\cap \Dp^\bullet = \emptyset$ or $B\cap \Dm^\bullet = \emptyset$.
	\end{enumerate} 
\end{definition}

Under the assumptions of Definition \ref{def:indefinc} there are positive constants $\betaL$ and $\betaU$ such that 
\begin{equation}
\betaL \leq \gamma_0 \leq \betaU \text{ in } \Omega,
\end{equation}
and there exists a large enough $\beta>0$ which satisfies
\begin{equation}
\sup (\kp) \leq \beta, \qquad \sup(\km) \leq \frac{\beta}{1+\beta}\betaL. \label{eq:kappabnd}
\end{equation}

\begin{remark}[Related to Definition~\ref{def:indefinc}]$ $\newline
	
	\begin{enumerate}[(i)]
		\item The condition $\sup(\km) < \inf(\gamma_0)$ is generally not required for the monotonicity method, and for homogeneous $\gamma_0$ (which is the default setting in \cite{Harrach13}) this is automatically satisfied as we must have $\gamma\in L^\infty_+(\Omega)$. In the proof of Theorem~\ref{thm:mono_continuum} it will be evident that this additional assumption, along with $\kpm\in L^\infty_+(\Omega)$, greatly simplifies the monotonicity method. It furthermore implies the existence of an efficient algorithm that implements the monotonicity method; see Section~\ref{sec:numexp}.
		\item While we do not immediately require additional regularity assumptions on $\gamma_0$, it should be noted that part of the results in Section~\ref{sec:mm} require that $\gamma_0$ satisfies a \emph{unique continuation property}. In Section~\ref{sec:rmm} we only require that an estimate \eqref{lamhbnd} holds for approximate models, though often more regularity is required from $\gamma_0$ to satisfy such an estimate (cf.\ \cite{GardeStaboulis_2016,Hyvonen09}). However, we do emphasize that the lack of regularity of the unknown perturbations $\kpm$ remains unchanged.
		\item Note that we allow either $\Dm$ or $\Dp$ to be empty, in which case the inclusions become definite. The results that follow still hold for definite inclusions, and the presented algorithms can also be used to find definite inclusions. The algorithms based on indefinite inclusions differ in a fundamental way from the typical monotonicity method for definite inclusions; the former uses arbitrary upper bounds on the inclusions, while the latter checks if specific open balls are inside the inclusions. These differences turn out to be significant when considering approximative models (cf.\ Section~\ref{sec:numexp}).
		\item The final technical assumption in Definition \ref{def:indefinc} is satisfied for most reasonable inclusions, and is mainly present to avoid some unusual pathological cases where both $\partial\Dp^\bullet$ and $\partial\Dm^\bullet$ coincide on a non-empty open subset of $\partial D^\bullet$. 
		
		One such example is the following, where $\Dp$ and $\Dm$ are unions of alternating open rings. Let $\Omega = B(0,2)$, open ball in $\rum{R}^d$, and consider four intertwining sequences $(s_j), (s_j'), (t_j), (t_j')\subset(0,1)$ satisfying
		\begin{align*}
		s_j &< s_j' < t_j < t_j' < s_{j+1}, \enskip \forall j, \\
		\lim_{j\to\infty} s_j &= \lim_{j\to\infty} s_j' = \lim_{j\to\infty} t_j = \lim_{j\to\infty} t_j' = 1.
		\end{align*}
		For each $j$ define
		\begin{align*}
		V_j &= \{x\in\rum{R}^d : s_j < \abs{x} < s_j'\}, \\
		W_j &= \{x\in\rum{R}^d : t_j < \abs{x} < t_j'\}.
		\end{align*}
		Now $\Dp = \cup_j V_j$ and $\Dm = \cup_j W_j$ are open and disjoint. However, in this case we have $\Dp^\bullet = \Dm^\bullet = \overline{B(0,1)}$, so the technical assumption is not satisfied at any point of $\partial D^\bullet$.
	\end{enumerate}
\end{remark}

The relation in \eqref{eq:simplemono} gives rise to a method of approximating the inclusion $D$ by checking for various upper bounds $C$ to $\Dp$ and $\Dm$; see Theorem~\ref{thm:mono_continuum}.(i). The reconstruction method was made precise in \cite{Harrach13} by also considering the converse of \eqref{eq:simplemono}. More precisely, counter examples to the monotonicity relation are proven using the theory of localized potentials when $C$ is not an upper bound of $D$; see \cite{Gebauer2008b} and Theorem \ref{thm:mono_continuum}.(ii). Below we formulate the results of \cite{Harrach13} in the setting of Definition~\ref{def:indefinc}. Due to the assumption that the perturbations $\kpm$ are essentially bounded away from zero, it is possible to state the result in a different way similar to what was done in \cite[Examples 4.8 and 4.10]{Harrach13}, but for general $L^\infty_+$-perturbations. The assumptions in Definition~\ref{def:indefinc} furthermore avoid the requirement that $\Dp$ and $\Dm$ should be well-separated similar to what was considered in \cite{Harrach13} when $\gamma-\gamma_0$ is assumed to be piecewise analytic.

The most important operators regarding the monotonicity method are
\begin{equation}\label{eq:mono-ops}
\begin{split}
\Tp(C) &= \Lambda(\gamma) - \Lambda(\gamma_0 + \beta\chi_C), \\ 
\Tp'(C) &= \Lambda(\gamma) - \Lambda(\gamma_0) - \beta\Lambda'(\gamma_0)\chi_C, \\
\Tm(C) &= \Lambda(\gamma_0 - \tfrac{\beta}{1+\beta}\betaL\chi_C) - \Lambda(\gamma),\\
\Tm'(C) &= \Lambda(\gamma_0) - \betaU\beta\Lambda'(\gamma_0)\chi_C - \Lambda(\gamma).
\end{split}
\end{equation}
Above notation for the parameter $\beta$ is suppressed as it is only required to be large enough to satisfy \eqref{eq:kappabnd}.

Theorem \ref{thm:mono_continuum} below characterizes the outer support $D^\bullet$ in terms of the operators in \eqref{eq:mono-ops}. Although the result is very close to \cite[Theorem 4.7 and 4.9]{Harrach13}, different assumptions imply that the equivalence condition is true for
\begin{enumerate}
	\item a fixed $\beta$-parameter which is easy to pick,
	\item a non-homogeneous background conductivity $\gamma_0$,
	\item non-smooth $L_+^\infty(\Omega)$ perturbations.
\end{enumerate}
Due to the listed differences, the proof is given for completion. In fact, the proof is also needed in Theorem \ref{thm:indep_mono} giving sufficient conditions for reconstructing $\Dp$ and $\Dm$ independently. 

In Theorem~\ref{thm:mono_continuum}.(ii) we assume that $\gamma_0$ satisfies a \emph{weak unique continuation property} (UCP) for open connected subsets $V\subseteq \Omega$. I.e.\ only the trivial solution of
\begin{equation*}
\nabla\cdot(\gamma_0\nabla v) = 0 \text{ in } V
\end{equation*}
can be identically zero on a non-empty open subset of $V$, and likewise only the trivial solution has vanishing Cauchy data on a non-empty open subset of $\partial V$. The coefficient $\gamma_0$ satisfies the UCP if e.g.\ it is Lipschitz continuous \cite[Section~19 and references within]{Miranda1970}. The assumption is required for the existence of localized potentials \cite{Gebauer2008b}. 

\begin{theorem} \label{thm:mono_continuum}
	Assume that $\gamma\in L^\infty_+(\Omega)$ is as in Definition \ref{def:indefinc} and let $\beta>0$ satisfy~\eqref{eq:kappabnd}. Then the following statements hold true.
	\begin{enumerate}[(i)]
		\item For any measurable $C\subseteq \overline{\Omega}$,
		\begin{align}
		\Dp \subseteq C \quad \text{implies} \quad &\Tp(C) \geq 0  \quad\text{and}\quad \Tp'(C) \geq 0, \label{Dp_mono_oneway}\\
		\Dm \subseteq C \quad \text{implies} \quad &\Tm(C) \geq 0  \quad\text{and}\quad \Tm'(C) \geq 0. \label{Dm_mono_oneway}
		\end{align}
		\item If $\gamma_0$ satisfies the UCP, then for any $C\in\mathcal{A}$, 
		\begin{align}
		D^\bullet\subseteq C\quad \text{if and only if} \quad \Tpm(C) \geq 0 \quad \text{if and only if} \quad \Tpm'(C) \geq 0. \label{D_mono}		
		\end{align}
	\end{enumerate}
\end{theorem}
{\em Proof of (i)}. The positive semi-definiteness of $\Tp(C)$ and $\Tm(C)$ in \eqref{Dp_mono_oneway} and \eqref{Dm_mono_oneway} is a direct consequence of the monotonicity relation \eqref{eq:simplemono} when $\Dp\subseteq C$ and $\Dm\subseteq C$, respectively. For $\Dp\subseteq C$, let $u_f$ denote the solution to \eqref{eq:cm} for Neumann condition $f\in L^2_\diamond(\partial\Omega)$ and conductivity $\gamma_0$. By inserting $\Tp'(C)$ into Proposition \ref{prop:monorel} and \eqref{eq:frechetdiff} and subsequently applying $\sup(\kp)\leq \beta$ and $\Dp\subseteq C$ gives
\begin{align*}
-\inner{\Tp'(C)f,f} &\leq \int_{\Omega} (\gamma-\gamma_0-\beta\chi_C)\abs{\nabla u_f}^2\,dx \\
&\leq \int_{\Omega} (\kp\chi_{\Dp}-\beta\chi_C)\abs{\nabla u_f}^2\,dx \\
&\leq -\beta\int_{\Omega} \chi_{C\setminus \Dp}\abs{\nabla u_f}^2\,dx \leq 0
\end{align*}
for all $f\in L^2_\diamond(\partial\Omega)$. This demonstrates  \eqref{Dp_mono_oneway}. Similarly $\Dm\subseteq C$ with Proposition~\ref{prop:monorel} yields
\begin{align*}
\inner{\Tm'(C)f,f} & \geq \int_{\Omega} \left[ \frac{\gamma_0}{\gamma}(\gamma-\gamma_0) + \betaU\beta\chi_C \right]\abs{\nabla u_f}^2\,dx \\
&\geq \int_{\Omega} \left[ \gamma_0\left(1-\frac{\gamma_0}{\gamma_0-\km\chi_{\Dm}}\right) + \betaU\beta\chi_C \right]\abs{\nabla u_f}^2\,dx \\
&= \int_{\Omega} \left[ \betaU\beta\chi_C - \frac{\gamma_0\km}{\gamma_0-\km}\chi_{\Dm} \right]\abs{\nabla u_f}^2\,dx \\
&\geq \betaU\beta\int_{\Omega} \chi_{C\setminus \Dm}\abs{\nabla u_f}^2\,dx \geq 0
\end{align*}
for all $f\in L^2_\diamond(\partial\Omega)$. In the bottom inequality the bounds in \eqref{eq:kappabnd} are applied through
\begin{equation}
\frac{\gamma_0\km}{\gamma_0-\km} \leq \frac{\betaU\tfrac{\beta}{1+\beta}\betaL}{\betaL - \tfrac{\beta}{1+\beta}\betaL} = \betaU\beta. \label{gammafracbnd}
\end{equation}
{\em Proof of (ii)}. The logical structure of this part of the proof can be outlined as 
\begin{equation*}
\begin{split}
P \Rightarrow P' &\text{ and } P \Rightarrow P'', \\
\neg P \Rightarrow \neg P' &\text{ and } \neg P \Rightarrow \neg P'',
\end{split}
\end{equation*}
where $P, P', P''$ are the conditions in the claimed equivalence relation in the same order from left to right as in the theorem statement.

Let $C\in\mathcal{A}$ be arbitrary. Then clearly $D\subset D^\bullet \subseteq C$ implies the inequalities $\Tpm(C)\geq 0$ and $\Tpm'(C)\geq 0$ by (i). 

Subsequently it suffices to contradict either one of the +/- inequalities if $D^\bullet\not\subseteq C$. Without loss of generality we may assume that $\Dp$ and $\Dm$ are non-empty. Recall that $\partial D^\bullet\subseteq \partial D$, so as $D$ is open and $C$ is closed implies that $D^\bullet\setminus C$ contains a relatively open subset which intersects $\partial D^\bullet\setminus C$. There are two possible (not mutually exclusive) cases:
\begin{enumerate}[(a)]
	\vspace{5\lineskip}
	\item $\partial\Dp^\bullet$ contains a non-empty open subset of $\partial D^\bullet\setminus C$,
	\item $\partial\Dm^\bullet$ contains a non-empty open subset of $\partial D^\bullet\setminus C$.
	\vspace{5 \lineskip}
\end{enumerate}
{\em Case (a):} Since $C$ and $D^\bullet$ have connected complement then we can pick a relatively open and connected set $U\subset\overline{\Omega}$ which intersects $\partial\Omega$ and has connected complement. Using the technical assumption in Definition~\ref{def:indefinc}, the set $U$ is chosen such that $U\cap C = \emptyset$, $U\cap \Dm=\emptyset$, and such that $U\cap \Dp$ contains an open ball $B$. This choice can be done as $\Dp$ is open and $C$ is closed; see Figure~\ref{fig:mono_proof_illu}. The main idea of the proof is to construct potentials $u$ where $\abs{\nabla u}^2$ is large enough inside $B$ and small enough outside $U$, such that the monotonicity principles in Proposition~\ref{prop:monorel} contradict the inequalities $\Tp(C)\geq 0$ and $\Tp'(C)\geq 0$.
\begin{figure}[htb]
	\centering
	\includegraphics[width=.5\textwidth]{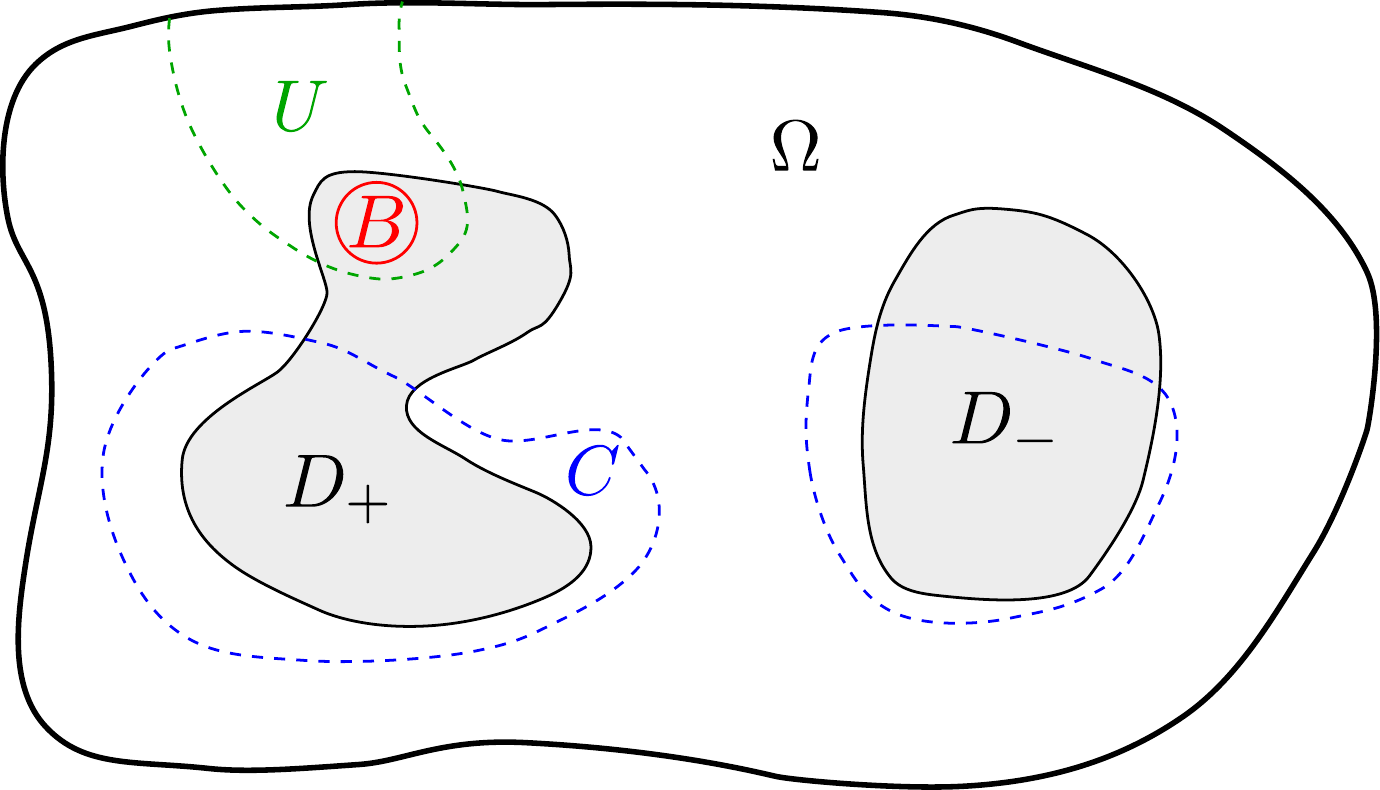}
	\caption{Illustration of case (a) in the proof of Theorem \ref{thm:mono_continuum}.(ii).\label{fig:mono_proof_illu}}
\end{figure}

We begin the argument by noting that $\gamma = \gamma_0+\kp$ on $B$ and $\gamma\geq \gamma_0$ on $U$. By the assumptions on $U$ and that $\overline{\Omega}$ has connected complement implies $(\overline{\Omega}\setminus U)^\bullet = \overline{\Omega}\setminus U$. Since $\gamma_0$ satisfies a UCP, then by \cite[Theorem 2.7]{Gebauer2008b} (see also \cite[Theorem 3.6]{Harrach13}) there exist sequences of current densities $(f_m^+)\subset L^2_\diamond(\partial\Omega)$ and the corresponding potentials $(u_m^{\gamma_0})\subset H_\diamond^1(\Omega)$, that solve \eqref{eq:cm} with conductivity $\gamma_0$, such that we have
\begin{equation}
\lim_{m\to\infty}\int_{B}\abs{\nabla u_m^{\gamma_0}}^2\,dx = \infty, \qquad \lim_{m\to\infty}\int_{\Omega\setminus U} \abs{\nabla u_m^{\gamma_0}}^2\,dx = 0. \label{locpot1}
\end{equation}
Such solutions are commonly referred to as \emph{localized potentials}.
Let $\hat{\gamma} = \frac{\gamma_0}{\gamma} (\gamma-\gamma_0)-\beta\chi_C$. By Proposition~\ref{prop:monorel}, \eqref{gammafracbnd}, and \eqref{locpot1}, we get
\begin{align}
-\inner{\Tp'(C)f_m^+,f_m^+} \hspace{-1.5cm}& \notag\\
&\geq \int_{\Omega} \hat{\gamma}\abs{\nabla u_m^{\gamma_0}}^2\,dx \notag\\
&= \int_B \frac{\gamma_0\kp}{\gamma_0+\kp}\abs{\nabla u_m^{\gamma_0}}^2\,dx  +\int_{U\setminus B} \frac{\gamma_0\kp\chi_{\Dp}}{\gamma_0+\kp\chi_{\Dp}}\abs{\nabla u_m^{\gamma_0}}^2\,dx + \int_{\Omega\setminus U}\hat{\gamma}\abs{\nabla u_m^{\gamma_0}}^2 dx \notag\\
&\geq \frac{\betaL\inf(\kp)}{\betaU+\beta}\int_B\abs{\nabla u_m^{\gamma_0}}^2\,dx - \beta(\betaU+1)\int_{\Omega\setminus U}\abs{\nabla u_m^{\gamma_0}}^2\,dx \to \infty \text{ for } m\to\infty. \label{Tp'_locpot}
\end{align}	
In other words, we have deduced the condition $\Tp'(C) \not\geq 0$. 

Since $C\subseteq \overline{\Omega}\setminus U$ then by \cite[Lemma 3.7]{Harrach13} there exists a sequence of localized potentials $(u_m^{\tilde{\gamma}})$ with corresponding current densities $(g^+_m)$ for the conductivity $\tilde{\gamma} = \gamma_0 + \beta\chi_C$. We denote $\hat{\gamma} = \frac{\tilde{\gamma}}{\gamma}(\gamma-\tilde{\gamma})$ and again apply Proposition~\ref{prop:monorel} with the asymptotic properties of the localized potentials to obtain
\begin{align}
-\inner{\Tp(C)g^+_m,g^+_m} \hspace{-1.5cm}& \notag\\
&\geq \int_{\Omega} \hat{\gamma}\abs{\nabla u_m^{\tilde{\gamma}}}^2\,dx \notag\\
&= \int_B \frac{\gamma_0\kp}{\gamma_0+\kp}\abs{\nabla u_m^{\tilde{\gamma}}}^2\,dx + \int_{U\setminus B} \frac{\gamma_0\kp\chi_{\Dp}}{\gamma_0 + \kp\chi_{\Dp}}\abs{\nabla u_m^{\tilde{\gamma}}}^2\,dx + \int_{\Omega\setminus U}\hat{\gamma}\abs{\nabla u_m^{\tilde{\gamma}}}^2\,dx \notag\\
&\geq \frac{\betaL\inf(\kp)}{\betaU+\beta}\int_B \abs{\nabla u_m^{\tilde{\gamma}}}^2\,dx + \inf(\hat{\gamma})\int_{\Omega\setminus U} \abs{\nabla u_m^{\tilde{\gamma}}}^2\,dx \to \infty \text{ for } m\to\infty. \label{Tp_locpot}
\end{align}
Notice that $\inf(\hat{\gamma}) > -\infty$ as $\gamma,\tilde{\gamma}\in L^\infty_+(\Omega)$. As a result of \eqref{Tp_locpot} we have $\Tp(C)\not\geq 0$.

\emph{Case (b):} In this case the roles of $\Dp$ and $\Dm$ are switched when choosing the sets $U$ and $B$, such that $\gamma = \gamma_0-\km$ on $B$ and $\gamma\leq \gamma_0$ on $U$.

Let $\hat{\gamma} = \gamma-\gamma_0+\betaU\beta\chi_C$. Using the localized potentials for $\gamma_0$ with current densities $(f^-_m)$ and the monotonicity relations in Proposition~\ref{prop:monorel}, we get
\begin{align}
\inner{\Tm'(C)f^-_m,f^-_m} \hspace{-1.5cm}& \notag\\
&\leq \int_\Omega \hat{\gamma}\abs{\nabla u_m^{\gamma_0}}^2\,dx \notag\\
& = -\int_B\km\abs{\nabla u_m^{\gamma_0}}^2\,dx - \int_{U\setminus B}\km\chi_{\Dm}\abs{\nabla u_m^{\gamma_0}}^2\,dx + \int_{\Omega\setminus U} \hat{\gamma}\abs{\nabla u_m^{\gamma_0}}^2\,dx \notag\\
&\leq -\inf(\km)\int_B\abs{\nabla u_m^{\gamma_0}}^2\,dx + \beta(\betaU + 1)\int_{\Omega\setminus U}\abs{\nabla u_m^{\gamma_0}}^2\,dx \to -\infty \text{ for } m\to\infty, \label{Tm'_locpot}
\end{align}
so $\Tm'(C)\not\geq 0$.

Finally, let $\tilde{\gamma} = \gamma_0 - \tfrac{\beta}{1+\beta}\betaL\chi_C$ and similar to case (a) we consider the localized potentials with respect to $\tilde{\gamma}$ with current densities $(g^-_m)$. From the monotonicity relations and properties of the localized potential we have
\begin{align}
\inner{\Tm(C)g^-_m,g^-_m} \hspace{-1.5cm}& \notag\\
&\leq \int_{\Omega} (\gamma-\tilde{\gamma})\abs{\nabla u_m^{\tilde{\gamma}}}^2\,dx \notag\\
&= -\int_B \km\abs{\nabla u_m^{\tilde{\gamma}}}^2\,dx - \int_{U\setminus B}\km\chi_{\Dm}\abs{\nabla u_m^{\tilde{\gamma}}}^2\,dx + \int_{\Omega\setminus U}(\gamma-\tilde{\gamma})\abs{\nabla u_m^{\tilde{\gamma}}}^2\,dx \notag\\
&\leq -\inf(\km)\int_B\abs{\nabla u_m^{\tilde{\gamma}}}^2\,dx + \tfrac{1+\beta+\betaL}{1+\beta}\beta\int_{\Omega\setminus U}\abs{\nabla u_m^{\tilde{\gamma}}}^2\,dx \to -\infty \text{ for } m\to\infty. \label{Tm_locpot}
\end{align}
Thus $\Tm(C)\not\geq 0$, which concludes the proof. \hfill$\square$

To facilitate the formulation of a reconstruction algorithm based on Theorem~\ref{thm:mono_continuum}.(ii) we define the following collections of test inclusions
\begin{alignat*}{2}
\Mp &= \{C\in\mathcal{A} : \Tp(C) \geq 0 \}, \qquad &\Mp' &= \{C\in\mathcal{A} : \Tp'(C) \geq 0 \}, \\
\Mm &= \{C\in\mathcal{A} : \Tm(C) \geq 0 \}, &\Mm' &= \{C\in\mathcal{A} : \Tm'(C) \geq 0 \}.
\end{alignat*}
Furthermore, we define $\M = \Mp\cap \Mm$ and $\M' = \Mp'\cap \Mm'$ as the collections of test inclusions that satisfy both inequalities.

\begin{corollary} \label{coro:recon}
	Under the assumptions of Theorem \ref{thm:mono_continuum}.(ii) there holds
	\begin{equation}
	D^\bullet = \cap \M = \cap\M'.
	\end{equation}
\end{corollary}
\begin{proof}
	By Theorem \ref{thm:mono_continuum}.(ii) we have $\M=\M'$ and for any $C\in\M$ we have $D^\bullet\subseteq C$. Since $D^\bullet \in \M$ the claim follows.
\end{proof}

By Corollary \ref{coro:recon} the sets $\M$ and $\M'$ characterize the outer shape of $D$. Therefore it makes sense to consider $\M$ and $\M'$ as theoretical reconstructions output by the monotonicity method. 

\begin{remark}
	Since no regularity assumptions are required for the inclusion boundaries $\partial \Dpm$, it is possible that $\abs{\partial D} > 0$ (in terms of Lebesgue measure in $\mathbb{R}^d$) for a sufficiently irregular boundary \cite{Osgood_1903}, and thus $\overline{D}$ may significantly differ from the shape of $D$; this is clearly a pathological case. 	
\end{remark}

Under the additional assumption that all of $\partial\Dp^\bullet$ and $\partial\Dm^\bullet$ can be connected to $\partial\Omega$ by only traversing $\overline{\Omega}\setminus D$, then the monotonicity method can be split into two independent submethods; one that determines the positive part and another that determines the negative part. 
\begin{theorem} \label{thm:indep_mono}
	In addition to the assumptions of Theorem \ref{thm:mono_continuum}.(ii), assume that the outer boundaries of $\Dp$ and $\Dm$ are well-separated, i.e.\ $\partial\Dp^\bullet\cap \partial\Dm^\bullet = \emptyset$.
	\begin{enumerate}[(i)]
		\item If $\Dp\cap \Dm^\bullet  = \emptyset$ then for any $C\in\mathcal{A}$ there holds
		\begin{equation}
		\Dp^\bullet\subseteq C\quad \text{if and only if} \quad \Tp(C) \geq 0\quad \text{if and only if} \quad \Tp'(C) \geq 0. \label{Dp_indepmono}
		\end{equation}
		In particular $\Dp^\bullet = \cap \Mp = \cap \Mp'$.
		\item If $\Dm\cap \Dp^\bullet  = \emptyset$ then for any $C\in\mathcal{A}$ there holds
		\begin{equation}
		\Dm^\bullet\subseteq C\quad \text{if and only if} \quad \Tm(C) \geq 0 \quad \text{if and only if} \quad \Tm'(C)\geq 0. \label{Dm_indepmono}
		\end{equation}
		In particular $\Dm^\bullet = \cap \Mm = \cap \Mm'$.
	\end{enumerate}
\end{theorem}
{\em Proof of (i):} Clearly $\Dp\subset \Dp^\bullet \subseteq C$ implies the inequalities $\Tp(C)\geq 0$ and $\Tp'(C)\geq 0$ in \eqref{Dp_indepmono} by Theorem \ref{thm:mono_continuum}.(i).

If $\Dp^\bullet\not\subseteq C$, then we can pick a relatively open set $U$ and open ball $B$ in the same way as in the proof of Theorem \ref{thm:mono_continuum}.(ii). The assumptions $\Dp\cap \Dm^\bullet = \emptyset$ and $\partial\Dp^\bullet\cap \partial\Dm^\bullet = \emptyset$ ensure that we can pick $U$ such that $U\cap \Dm = \emptyset$, and they are required to enable case (a) in the proof of Theorem \ref{thm:mono_continuum}.(ii). Hence, \eqref{Tp'_locpot} and \eqref{Tp_locpot} contradict the inequalities in \eqref{Dp_indepmono}. The claim $\Dp^\bullet = \cap\Mp = \cap\Mp'$ is proved in the same way as Corollary \ref{coro:recon}. 

{\em Proof of (ii):} The proof for $\Dm$ is similar to (i); when $\Dm^\bullet\not\subseteq C$ then the inequalities in \eqref{Dm_indepmono} are contradicted by \eqref{Tm'_locpot} and \eqref{Tm_locpot}. \hfill$\square$

\begin{remark}
	If, for instance, the assumptions of Theorem \ref{thm:indep_mono}.(i) hold, but $\Dm\cap\Dp^\bullet\neq \emptyset$, then $D^\bullet$ and $\Dp^\bullet$ can be reconstructed but $\Dm^\bullet$ cannot. However, it is still possible to obtain information about $\Dm$ using the set $D^\bullet\setminus \Dp^\bullet\subseteq \Dm^\bullet$ which corresponds to the part of $\Dm$ that can be connected to $\partial\Omega$ without crossing $\Dp$.
\end{remark}

\section{Regularizing the monotonicity method} \label{sec:rmm}

In this section we formulate the regularized monotonicity method for indefinite inclusions. The presented construction is analogous to the definite case treated in \cite{GardeStaboulis_2016}. 

In real-life measurements from EIT-devices there are two main sources of error: 
\begin{enumerate}[1.]
	\item \emph{Modelling error} characterized by a parameter $h$, which states how closely related the actual measurements (or realistic electrode models) are to the CM.
	\item \emph{Noise error} controlled by a parameter $\delta > 0$. This error is caused by imperfections of the measurement device and it is not directly related to the model. 
\end{enumerate}
By symmetrizing if necessary, the noise error can be modelled as a self-adjoint operator $N^\delta \in\mathcal{L}(L^2_\diamond(\partial\Omega))$
which satisfies $\norm{N^\delta}_{\mathcal{L}(L^2_\diamond(\partial\Omega))} \leq \delta$. Note that, in contrast to \cite{GardeStaboulis_2016}, here we have dropped the unnecessary assumption which requires $N^\delta$ to be compact. 

To describe the modelling error we consider a family of compact and self-adjoint approximative operators $\{ \Lambda_h(\gamma) \}_{h>0} \subset \mathcal{L}(L^2_\diamond(\partial\Omega))$, such that the following estimate holds for each $\gamma\in L^\infty_+(\Omega)$,
\begin{equation}
\norm{\Lambda(\gamma)-\Lambda_h(\gamma)}_{\mathcal{L}(L^2_\diamond(\partial\Omega))} \leq \omega(h)\norm{\gamma}_{L^\infty(\Omega)},\quad \lim_{h\to 0}\omega(h)=0. \label{lamhbnd}
\end{equation}
The measurements are now modelled with additive noise
\begin{equation}
\Lambda_h^\delta(\gamma) = \Lambda_h(\gamma) + N^\delta.
\end{equation}

To adapt the monotonicity method with the approximative operators $\Lambda_h$ and noisy measurements, we denote
\begin{alignat*}{3}
\Tph(C) &= \Lambda_h(\gamma) - \Lambda_h(\gamma_0+\beta\chi_C), \qquad& \Tmh(C) &=  \Lambda_h(\gamma_0-\tfrac{\beta}{1+\beta}\betaL\chi_C) - \Lambda_h(\gamma),\\
\Tphd(C) &= \Lambda_h^\delta(\gamma) - \Lambda_h(\gamma_0+\beta\chi_C),  & \Tmhd(C) &= \Lambda_h(\gamma_0-\tfrac{\beta}{1+\beta}\betaL\chi_C)-\Lambda_h^\delta(\gamma),
\end{alignat*}
for any measurable $C\subseteq \overline{\Omega}$ (cf.\ Section \ref{sec:mm}). By \eqref{lamhbnd} and \cite[Lemma 1]{GardeStaboulis_2016} the difference in infima of the spectra of $\Tpmhd$ and $\Tpm$ satisfies the bound condition
\begin{equation*}
\abs{ \inf\sigma(\Tpmhd(C)) - \inf\sigma(\Tpm(C)) } \leq \omega(h)\left(\norm{\gamma_0}_{L^\infty(\Omega)} + \norm{\gamma}_{L^\infty(\Omega)} + \max\left\{\beta, \betaL \right\} \right) + \delta
\end{equation*}
implying that
\begin{equation}
\lim_{h,\delta\to 0}\inf\sigma(\Tpmhd(C)) = \inf\sigma(\Tpm(C)) \quad\text{uniformly in}\quad C\subseteq\overline{\Omega}. \label{uniconvinf}
\end{equation}

At this point we define the regularized reconstruction as $\cap \mathcal{M}_{\ap,\am}(\Tphd,\Tmhd)$ where
\begin{align*}
\M_{\ap,\am}(S_{+},S_{-}) &= \{ C\in\mathcal{A}: S_{\pm}(C) + \apm(C)\ident\geq 0\}
\end{align*}
and $\apm:\mathcal{A}\to\mathbb{R}$ are regularization parameters which may depend on $C\in\mathcal{A}$ (see Remark \ref{remark:thm_conv}), and $S_{\pm}(C)\in \mathcal{L}(L^2_\diamond(\partial\Omega))$ is self-adjoint for each measurable $C\subseteq \overline{\Omega}$. In addition, we denote $\M_{\alpha}(S) = \M_{\alpha,\alpha}(S,S)$. In particular, the reconstruction methods for the CM are included in this notation via $\M = \M_{0,0}(\Tp,\Tm)$ and $\Mpm = \M_{0}(\Tpm)$.

The following result establishes the convergence of the regularized monotonicity method, in the limit as the modelling error and noise error tend to zero. 
\begin{theorem} \label{thm:conv} 
	Suppose that the regularization parameters $\apm = \apm(h,\delta)[\cdot] : \mathcal{A}\to\mathbb{R}$ are bounded by 
	\begin{equation*}
	\apmL(h,\delta) \leq \apm(h,\delta)[C] \leq \apmU(h,\delta)
	\end{equation*}
	for all $C\in\mathcal{A}$ and $h,\delta>0$, where $\apm^\textup{L}$ and $\apm^\textup{U}$ satisfy
	\begin{equation}
	\delta-\apmL(h,\delta) \leq \inf_{C\in\M}\inf\sigma(\Tpmh(C)) \quad \text{and} \quad \lim_{h,\delta\to 0}\apmU(h,\delta) = 0. \label{eq:alphabnds}
	\end{equation}
	Then for any $\lambda>0$ there exists an $\epsilon_\lambda>0$ such that
	\begin{equation}
	\M \subseteq \M_{\ap(h,\delta),\am(h,\delta)}(\Tphd,\Tmhd) \subseteq \M_{\lambda,\lambda}(\Tp,\Tm), \label{mrecon}
	\end{equation}
	for all $h,\delta\in(0,\epsilon_\lambda]$. 
	
	Replacing $\M$ by $\Mpm$ in \eqref{eq:alphabnds} yields
	\begin{equation}
	\Mpm \subseteq \M_{\apm(h,\delta)}(\Tpmhd) \subseteq \M_{\lambda}(\Tpm), \label{mpmrecon}
	\end{equation}
	for all $h,\delta\in(0,\epsilon_\lambda]$.
\end{theorem}
\begin{proof}
	We only prove \eqref{mrecon} since the adaptation of the proof for \eqref{mpmrecon} is straightforward. First we prove the left-hand set inclusion in \eqref{mrecon}. Let $C\in \M$ be arbitrary, then
	\begin{equation}
	\Tpmh(C) \geq \inf\sigma(\Tpmh(C))\ident \geq (\delta-\apmL(h,\delta))\ident. \label{eq:Tpmhbnd}
	\end{equation}
	Since $\delta\ident\geq \pm N^\delta$ then \eqref{eq:Tpmhbnd} implies
	\begin{equation*}
	\Tpmhd(C) + \apm(h,\delta)[C]\ident = \Tpmh(C) \pm N^\delta + \apm(h,\delta)[C]\ident \geq \delta\ident\pm N^\delta\geq 0.
	\end{equation*}
	Thus $C\in\M_{\ap(h,\delta),\am(h,\delta)}(\Tphd,\Tmhd)$ for all $h,\delta>0$.
	
	To prove the right-hand set inclusion in \eqref{mrecon}, let $\lambda_+>0$ and $\lambda_->0$ be arbitrary. The limit $\lim_{h,\delta\to 0}\apmU(h,\delta) = 0$ and \eqref{uniconvinf} implies that there exist $\epsilon_{\lambda_+}>0$ and $\epsilon_{\lambda_-}>0$ such that
	\begin{equation*}
	\apm(h,\delta)[C] \leq \frac{\lambda_{\pm}}{2}, \quad \inf\sigma(\Tpmhd(C))\leq \inf\sigma (\Tpm(C)) + \frac{\lambda_{\pm}}{2}
	\end{equation*}
	for all $h,\delta\in(0,\epsilon_{\lambda_{\pm}}]$ and $C\in\mathcal{A}$. Denote $\lambda = \max\{\lambda_+,\lambda_-\}$ and $\epsilon_\lambda = \min\{\epsilon_{\lambda_+},\epsilon_{\lambda_-}\}$. For any $C\in \M_{\ap(h,\delta),\am(h,\delta)}(\Tphd,\Tmhd)$ we have
	\begin{align*}
	0 &\leq \inf\sigma(\Tpmhd(C)) + \apm(h,\delta)[C] \leq \inf\sigma(\Tpm(C))+\lambda_{\pm} \leq \inf\sigma(\Tpm(C))+\lambda,
	\end{align*}
	which implies that $\M_{\ap(h,\delta),\am(h,\delta)}(\Tphd,\Tmhd)\subseteq \M_{\lambda,\lambda}(\Tp,\Tm)$ for all $h,\delta\in (0,\epsilon_\lambda]$.
\end{proof}

Note that $0<\lambda_{\pm}\leq \mu_{\pm}$ implies $\M_{\lambda_+,\lambda_-}(\Tp,\Tm) \subseteq \M_{\mu_+,\mu_-}(\Tp,\Tm)$, thus by monotonicity we have the set-theoretic limit
\begin{equation}
\lim_{\lambda\to 0} \M_{\lambda,\lambda}(\Tp,\Tm) = \bigcap_{\lambda>0}\M_{\lambda,\lambda}(\Tp,\Tm) = \M,
\end{equation}
and likewise $\lim_{\lambda\to 0} \M_{\lambda}(\Tpm) = \Mpm$. The sets $\M_{\ap,\am}(\Tphd,\Tmhd)$ are not guaranteed to be monotone. However from \eqref{mrecon} and the squeeze-principle we obtain the existence of a set-theoretic limit that coincides with $\M$.
\begin{corollary} \label{coro:limit}
	Let the regularization parameters be as in Theorem \ref{thm:conv}. Then we have the set-theoretic limits
	\begin{align*}
	\lim_{h,\delta\to 0} \M_{\ap(h,\delta),\am(h,\delta)}(\Tphd,\Tmhd) &= \M = \bigcap_{h,\delta>0} \M_{\ap(h,\delta),\am(h,\delta)}(\Tphd,\Tmhd), \\
	\lim_{h,\delta\to 0} \M_{\apm(h,\delta)}(\Tpmhd) &= \Mpm = \bigcap_{h,\delta>0} \M_{\apm(h,\delta)}(\Tpmhd).
	\end{align*}
\end{corollary}
\begin{proof}
	The proof is analogous to the proof in the definite case \cite[Corollary 1]{GardeStaboulis_2016}.	
\end{proof}

To conclude the section, we give a few remarks related to extensions of the regularized monotonicity method.

\newpage
\begin{remark}[Related to Theorem \ref{thm:conv} and Corollary \ref{coro:limit}] \label{remark:thm_conv} $ $
	\begin{enumerate}
		\item \textbf{\textup{Regularized linear monotonicity method:}} Define $(\Tpm')^h$ and $(\Tpm')^{h,\delta}$ as in Section \ref{sec:mm} using the approximative models $\{\Lambda_h\}_{h>0}$ and with noisy measurement $\Lambda_h^\delta(\gamma)$.\ Then Theorem \ref{thm:conv} and Corollary \ref{coro:limit} can be formulated using the linear monotonicity method by appropriately replacing $\M$, $\Mpm$, $\Tpm$, $\Tpmh$, and $\Tpmhd$ by $\M'$, $\Mpm'$, $\Tpm'$, $(\Tpm')^h$, and $(\Tpm')^{h,\delta}$, respectively. The only additional assumption is that $\Lambda_h$ is Fr\'echet differentiable with $\Lambda_h'$ that for each $\gamma\in L^\infty_+(\Omega)$ and $\eta\in L^\infty(\Omega)$ satisfies the estimate
		\begin{equation}
		\norm{\Lambda'(\gamma)\eta - \Lambda_h'(\gamma)\eta}_{\mathcal{L}(L^2_\diamond(\partial\Omega))} \leq \omega(h)\norm{\gamma}_{L^\infty(\Omega)}\norm{\eta}_{L^\infty(\Omega)}, \label{lamhdiffbnd}
		\end{equation}
		which gives the corresponding uniform convergence in \eqref{uniconvinf} for $(\Tpm')^{h,\delta}$ and $\Tpm'$.
		
		\item \textbf{\textup{Regularization parameters:}} 
		\begin{enumerate}[(i)]
			\item Similar to \cite{GardeStaboulis_2016} the regularization parameters can be chosen as 
			\begin{equation*}
			\apm(h,\delta)[C] = \apL(h,\delta) = \apU(h,\delta),
			\end{equation*}
			i.e.\ independent of the set $C$ (and the set of such admissible regularization parameters is non-empty). The reason for introducing the additional flexibility that $\apm$ can depend on $C$, is because for a fixed $h,\delta>0$ it turns out that varying the regularization depending on $\abs{C}$ yields better numerical results. This was not observed in the definite case \cite{GardeStaboulis_2016} since in those algorithms the numerical examples could be computed by checking monotonicity relations for balls of a fixed radius.
			\item Assume that $\Lambda_h$ satisfies the simple monotonicity relation \eqref{eq:simplemono}, i.e.\ 
			\begin{equation}
			\gamma\geq \tilde{\gamma}\quad \text{implies}\quad \Lambda_h(\tilde{\gamma})\geq \Lambda_h(\gamma), \label{eq:simplemonoh}
			\end{equation}
			for all $\gamma,\tilde{\gamma}\in L^\infty_+(\Omega)$ and $h>0$. By choosing $\beta>0$ such that \eqref{eq:kappabnd} is satisfied, then from Theorem \ref{thm:mono_continuum}.(ii) and the definition of $\M$, it holds
			\begin{equation*}
			\Tpmh(C) \geq \Tpmh(D^\bullet) \geq 0, \enskip C\in\M.
			\end{equation*}
			As a consequence, since $\Tpmh(C)$ is compact and self-adjoint, it is sufficient to have $\apmL(h,\delta) \geq \delta$ in \eqref{eq:alphabnds}. The same conclusion holds for the regularization parameters in \eqref{mpmrecon} if the assumptions of Theorem \ref{thm:indep_mono} are satisfied.
		\end{enumerate}
		\item \textbf{\textup{Relation to the definite case:}} From Theorem \ref{thm:conv} it is clear that if we consider the reconstruction to be the intersection $\cap \M_{\ap,\am}(\Tphd,\Tmhd)$, then we obtain a lower bound to the noise-free CM reconstruction,
		\begin{equation*}
		\cap \M_{\ap,\am}(\Tphd,\Tmhd) \subseteq \cap \M.
		\end{equation*}
		This is in contrast to the definite case in \cite{GardeStaboulis_2016} where an upper bound to $D$ is obtained. Consider a definite case where $\gamma = \gamma_0 + \kp\chi_{\Dp}$ and $\overline{\Dp}\subseteq\Omega$ has a connected complement. Then we can guarantee a lower bound of $\Dp$ using the method for the indefinite case and an upper bound using the method for the definite case in \cite{GardeStaboulis_2016}.  Moreover, the bounds converge from below and above, respectively, as both the noise and modelling error tend to zero.
		\item \textbf{\textup{Including prior knowledge:}} If a lower bound on the volume $t\leq \abs{D^{\bullet}}$ of the inclusions is known, then we can consider  
		\begin{equation*}
		\M^t_{\ap,\am}(\Tphd,\Tmhd) \subseteq \M_{\ap,\am}(\Tphd,\Tmhd)
		\end{equation*}
		as the subset for which $\abs{C}\geq t$. The proof of Theorem \ref{thm:conv} directly adapts to such prior information, where similarly we have $\M^t_{\lambda,\lambda}(\Tp,\Tm)$ on the right-hand side in \eqref{mrecon}.\ Then we obtain tighter reconstructions the closer the bound is to $\abs{D^\bullet}$, i.e.\ for $t_1\leq t_2 \leq \abs{D^\bullet}$,
		\begin{equation*}
		\cap \M^{t_1}_{\ap,\am}(\Tphd,\Tmhd) \subseteq \cap \M^{t_2}_{\ap,\am}(\Tphd,\Tmhd) \subseteq  \cap\M.
		\end{equation*}
		\item \textbf{\textup{Operator estimates:}} In the estimate \eqref{lamhbnd} the dependence on $\gamma$ only through $\norm{\gamma}_{L^\infty(\Omega)}$ is needed for the uniform convergence in \eqref{uniconvinf}. Thus, the estimate \eqref{lamhbnd} can be relaxed, by having $\omega(h)q(\norm{\gamma}_{L^\infty(\Omega)})$ on the right-hand side with $q : [0,\infty)\to[0,\infty)$ a non-decreasing function. Furthermore, the estimate is also only required to hold for $\gamma$ in Definition \ref{def:indefinc}, $\gamma_0 + \beta\chi_C$, and $\gamma_0 - \frac{\beta}{1+\beta}\betaL\chi_C$ with $C\in\mathcal{A}$.  
		\item \textbf{\textup{Formulation with pseudospectra \cite{Trefethen2005}:}} The $\delta$-pseudo\-spec\-trum of a closed operator $A$ on a Banach space $X$ is defined as
		\begin{equation*}
		\sigma_\delta(A) = \{ z\in \sigma(A+B) : B\in\mathcal{L}(X),\ \norm{B} < \delta \}.
		\end{equation*}
		For $\norm{N^\delta}_{\mathcal{L}(L^2_\diamond(\partial\Omega))} < \delta$, the inequality $\Tpmhd(C) + \apm(C)\ident\geq 0$ is therefore related to the pseudospectrum of $\Tpmh(C)$ by 
		\begin{equation*}
		\sigma(\Tpmhd(C) + \apm(C)\ident) \subseteq \apm(C) + \sigma_{\delta}(\Tpmh(C)).
		\end{equation*}
		Since $\cap_{\delta>0}\sigma_\delta(\Tpmh(C)) = \sigma(\Tpmh(C))$ \cite[Theorem 4.3]{Trefethen2005}, one can partly understand the limit in Corollary \ref{coro:limit} in the sense of pseudospectra. More discussions on this topic are left for future studies.
	\end{enumerate}
\end{remark}

Before discussing implementation details and numerical examples, we note that the CEM is justified as an approximative model for the monotonicity method (cf.~\cite[Theorems 2 and 3]{GardeStaboulis_2016} and \cite{Hyvonen09}) using the concept of artificially extended electrodes. The modelling error $h$ is related to how densely $\partial\Omega$ is covered by small electrodes. Furthermore the current-to-voltage maps, and their linearizations, from the CEM can be directly used in the monotonicity relations without any modifications \cite[Proposition 4]{GardeStaboulis_2016}. Thereby the monotonicity method for the CEM simply replaces the ND maps with the corresponding matrix current-to-voltage maps.

The CEM is defined in \eqref{eq:cem1}--\eqref{eq:cem4}, where $\{E_j\}_{j=1}^k$ denote non-empty open subsets of $\partial\Omega$, that model the position of $k$ physical surface electrodes, and $z_j>0$ will denote the contact impedance on the $j$th electrode. Here $\{\overline{E_j}\}_{j=1}^k$ are assumed to be mutually disjoint. For an input net current pattern $I$ belonging to the hyperplane
\begin{equation*}
\mathbb{R}^k_\diamond = \left\{ W\in\mathbb{R}^k : \sum_{j=1}^k W_j = 0 \right\},
\end{equation*}
there exists a unique solution $(v,V)\in H^1(\Omega)\oplus \mathbb{R}_\diamond^k$ to 
\begin{alignat}{2}
\nabla\cdot(\gamma\nabla v) &= 0, \quad&&\text{in } \Omega, \label{eq:cem1} \\
\nu\cdot\gamma\nabla v &= 0, &&\text{on } \partial\Omega \setminus \cup_{j=1}^k\overline{E_j}, \\
v + z_j\nu\cdot \gamma\nabla v &= V_j, &&\text{on } E_j, \\
\int_{E_j} \nu\cdot\gamma\nabla v\,dS &= I_j, &&j=1,2,\dots,k. \label{eq:cem4} 
\end{alignat}
Here $v$ is the interior electric potential and $V$ comprise the electrode voltages. The corresponding current-to-voltage map for the CEM is the symmetric matrix $R(\gamma) : \mathbb{R}^k_\diamond\to \mathbb{R}^k_\diamond,\enskip I\mapsto V$.

We note that there is also a recent computational relaxation to the CEM, the smoothened CEM \cite{Hyvonen2017}, which has favourable regularity properties that allow the use of higher order finite element methods. 

\section{Implementation details and numerical examples} \label{sec:numexp}

In this section we outline a peeling-type algorithm for implementing the regularized monotonicity method for indefinite inclusions. The main idea is to initialize with a set $C\in\mathcal{A}$, represented in a pixel discretization, that contains $D$. Afterwards, we successively remove pixels from the boundary of the discretization of $C$ if the new smaller set yields positive semi-definiteness in the monotonicity tests. That is, we peel layers from $C$ while maintaining that $C\in\mathcal{A}$ and yields positive semi-definiteness in the monotonicity tests. 

We clearly have
\begin{equation}
\M_{\apL,\amL}(\Tphd,\Tmhd) \subseteq \M_{\ap,\am}(\Tphd,\Tmhd) \subseteq \M_{\apU,\amU}(\Tphd,\Tmhd).
\end{equation}
Thus, if the presented algorithm is meaningful for regularization parameters independent of $C\in\mathcal{A}$, i.e.\ $\apmL$ and $\apmU$, then the squeeze-principle and Corollary \ref{coro:limit} imply that the algorithm will still converge in the limit as $h,\delta\to 0$ when using regularization parameters depending on $C$. 

To relate the algorithm to the theory presented in Section \ref{sec:rmm} we assume that $\Lambda_h$ satisfies the simple monotonicity relation \eqref{eq:simplemonoh}, which is true for the CEM. Definition~\ref{def:indefinc} implies that a single fixed $\beta$-value can be used for the monotonicity tests for any $C\in\mathcal{A}$, i.e.\ we have that $\gamma_0-\tfrac{\beta}{1+\beta}\betaL\chi_C$ is in $L^\infty_+(\Omega)$ for any $C\in\mathcal{A}$. Combined with the monotonicity relation, then
\begin{align*}
\M_{\apL,\amL}(\Tphd,\Tmhd) \ni C_1\subseteq C_2\in\mathcal{A} \quad &\text{implies} \quad C_2\in \M_{\apL,\amL}(\Tphd,\Tmhd), \\
\M_{\apU,\amU}(\Tphd,\Tmhd) \ni C_1\subseteq C_2\in\mathcal{A} \quad &\text{implies} \quad C_2\in \M_{\apU,\amU}(\Tphd,\Tmhd).
\end{align*}
Hence, it is meaningful to initialize with a large set $C\in\mathcal{A}$ and shrink it as long as it yields positive semi-definiteness in the monotonicity relations.

\begin{remark}
	Note that without the assumption $\sup(\km) < \inf(\gamma_0)$ in Definition~\ref{def:indefinc}, which allows the use of a fixed $\beta>0$, there is the possibility that $\beta$ can only be chosen large enough for the monotonicity relations when $C$ is a slight upper bound to $D$; see Figure~\ref{fig:badgamma0}. In that case there is no obvious efficient algorithm for the reconstruction, and one must check monotonicity relations for all $C\in\mathcal{A}$.
\end{remark}
\begin{figure}[htb]
	\centering
	\includegraphics[width=.6\textwidth]{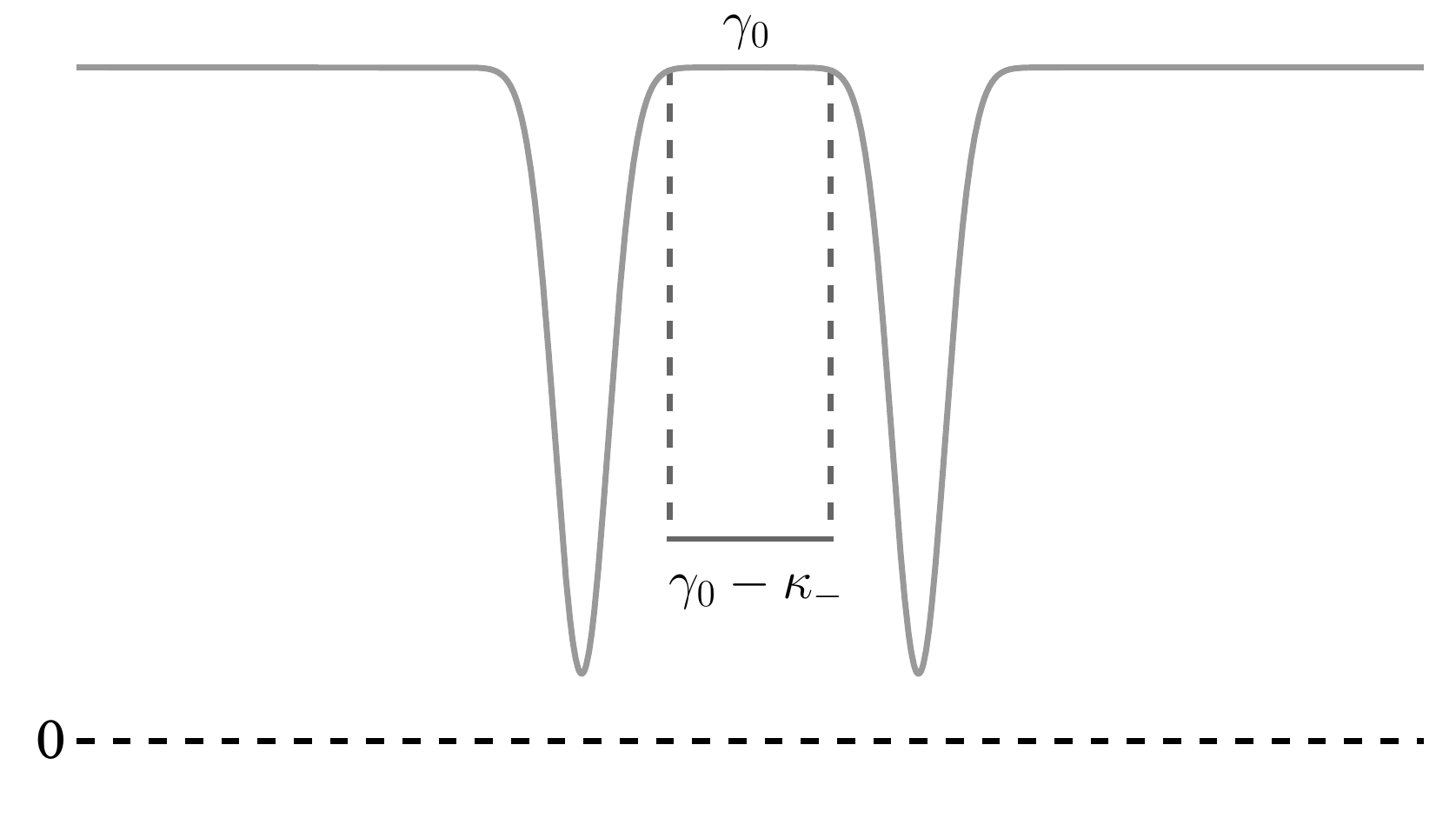}
	\caption{Profile of a conductivity distribution $\gamma$ which does not satisfy the assumption $\sup(\km)<\inf(\gamma_0)$.\label{fig:badgamma0}}
\end{figure}

To describe the implementation details we consider a regular pixel discretization $\{p_j\}_{j\in \mathcal{X}}$ of $\overline{\Omega}$ (voxel discretization in dimension three), where $\mathcal{X}$ is a finite index set and $p_j\subseteq\overline{\Omega}$ with pixel size $\abs{p_j} = \rho>0$. Thus, a set $C\in\mathcal{A}$ is represented as $C_{\mathcal{J}} = \cup_{j\in\mathcal{J}}p_j$ where $\mathcal{J}\subseteq \mathcal{X}$ denotes an index set with the pixel indices. To handle the bookkeeping involved in the algorithm, for a given reconstruction candidate $C_\mathcal{J}$, we consider the three index sets $\mathcal{J},\mathcal{I},\mathcal{Y}\subseteq \mathcal{X}$ defined as follows:
\begin{itemize}
	\item $\mathcal{J}$: indices of pixels in the set $C_\mathcal{J}$.
	\item $\mathcal{I}$: indices of boundary pixels of $C_\mathcal{J}$ that are yet to be tested if they can be removed.
	\item $\mathcal{Y}$: indices of pixels that have already been tested.
\end{itemize}  
Hence, the algorithm can be described by the pseudocode in Algorithm~\ref{alg1}, where \emph{neighbour pixels} refer to the closest four neighbours (NWSE) in two dimensions and the closest six neighbours in three dimensions, i.e.\ not diagonal pixels.
\begin{algorithm} \caption{Monotonicity-based reconstruction of indefinite inclusions} \label{alg1}
	\begin{algorithmic}[1]
		\LineComment{initialize}
		\State $\mathcal{J} := \{ j\in\mathcal{X} : j \text{ is a pixel index from the initial given $C_\mathcal{J}$} \}$ 
		\State $\mathcal{I} := \{ j\in\mathcal{J} : j \text{ is a boundary pixel of $C_\mathcal{J}$} \}$ 
		\State $\mathcal{Y} := \mathcal{X}\setminus\mathcal{J}$
		\For{$i\in\mathcal{I}$}
		\LineCommentNewEnv{test if pixel $i$ can be removed}
		\State $\tilde{\mathcal{J}} := \mathcal{J}\setminus\{i\}$,\quad $\mathcal{I} := \mathcal{I}  \setminus\{i\}$,\quad $\mathcal{Y} := \mathcal{Y}\cup\{i\}$
		\State Perform regularized monotonicity tests for $C_{\tilde{\mathcal{J}}}$
		\If{monotonicity tests are positive semi-definite}
		\LineCommentNewEnv{remove pixel $i$ from reconstruction}
		\State $\mathcal{J} := \tilde{\mathcal{J}}$
		\LineComment{add new boundary pixels of $C_{\mathcal{J}}$ to be tested}
		\State $\mathcal{I} := \mathcal{I} \cup \{ m\in \mathcal{J}\setminus\mathcal{Y} : m \text{ neighbour pixel to } i \}$
		\EndIf
		\EndFor
		\LineComment{the algorithm terminates when $\mathcal{I} = \emptyset$}
		\State Return reconstruction $C_\mathcal{J}$
	\end{algorithmic}
\end{algorithm}

Arguably, the algorithm for the indefinite case is more complicated than for the definite case in \cite{GardeStaboulis_2016} since we are required to keep track of the inclusion boundary in each step. For the computational complexity, however, the methods are equivalent. Depending on the size of $D$ the indefinite case may actually be significantly faster as the algorithm does not have to perform monotonicity tests for all the pixels in the domain, unlike the definite case. Furthermore, the algorithm can be initialized by attempting to remove all the boundary pixels of $C_\mathcal{J}$ simultaneously, and transition to removing single pixels when it fails to remove the whole boundary. 

For the numerical examples we will use the linear version of the monotonicity method for the CEM with contact impedance $z = 10^{-2}$ on all electrodes; numerical tests with the definite case suggests that the linear and nonlinear methods are equally noise robust \cite{Garde_2017a}. Furthermore, we will only give reconstructions based on $\Mp$ and $\Mm$ such that only a single regularization parameter is required. As mentioned in Remark~\ref{remark:thm_conv} it is sufficient to have $\apmL \geq \delta$. As the regularization parameter for a fixed noise level $\delta>0$ we use
\begin{equation}
\alpha(C) = \frac{\alpha_0}{\abs{C}}, \label{eq:regparnum}
\end{equation}
where $\alpha_0 \geq \abs{\Omega}\delta$ is to be tuned. We assume that at least one pixel will be in the reconstruction such that $0<\rho\leq\abs{C}\leq \abs{\Omega}$ for the considered test inclusions, thus satisfying \eqref{eq:alphabnds}. The choice \eqref{eq:regparnum} implies that smaller sets $C\in\mathcal{A}$ are regularized more in the monotonicity tests; this choice is not known to be optimal, but it did improve upon the reconstructions compared to using a uniform parameter. We also hypothesize that the distance of $C$ to the boundary has an important role in an optimal choice of regularization \cite{Garde_2017b}. However, we leave a detailed analysis regarding regularization parameter choices for future studies. 

In the numerical examples we use the orthonormal set of current patterns $\mathbf{I} = [I^{(1)}, I^{(2)},\dots, I^{(k-1)}]$ defined by
\begin{equation}
I_j^{(m)} = \begin{cases}
\sqrt{\frac{1}{m(m+1)}} & j = 1,2,\dots,m, \\
-\sqrt{\frac{m}{m+1}} & j = m+1, \\
0 & j = m+2,m+3,\dots,k. 
\end{cases} \label{eq:currentpatt}
\end{equation}
Note that \eqref{eq:currentpatt} is the Gram--Schmidt orthonormalization of the standard basis $\{e^{(1)}-e^{(m+1)}\}_{m=1}^{k-1}$ for $\mathbb{R}_\diamond^k$.

In each example we simulate $V^{(m)} = R(\gamma)I^{(m)}$, and collect the output as
\begin{equation*}
\mathbf{V} = [V^{(1)},V^{(2)},\dots,V^{(k-1)}].
\end{equation*}
To simulate noise, define the $k\times k$ matrix $\mathbf{Y}$ with $\mathbf{Y}_{i,j}$ drawn from a normal $\mathcal{N}(0,1)$-distribution. Now we consider the perturbed measurements
\begin{equation*}
\tilde{\mathbf{V}} = \mathbf{H}({\ident}_{k} + \mathbf{Y})\mathbf{V},
\end{equation*}
where ${\ident}_k$ is the $k\times k$ identity matrix and $\mathbf{H} = {\ident}_k-\frac{1}{k}\mathbf{1}_k$ is the orthogonal projection of $\mathbb{R}^k$ onto $\mathbb{R}^k_\diamond$, i.e.\ $\mathbf{1}_k$ is the $k\times k$ matrix with $1$ in each entry. Note that $\mathbf{R} = \mathbf{I}^{\textup{T}}\mathbf{V}$ is the matrix representation of $R(\gamma)$ in the $\mathbf{I}$-basis, thus
\begin{equation*}
\sym(\mathbf{I}^{\textup{T}}\tilde{\mathbf{V}}) = \mathbf{R} + \sym(\mathbf{I}^{\textup{T}}\mathbf{H}\mathbf{Y}\mathbf{V})
\end{equation*}
where $\sym$ is the symmetric part of a matrix. We therefore write the noisy measurement as $\mathbf{R}^\delta = \mathbf{R} + \mathbf{N}^\delta$ with
\begin{equation*}
\mathbf{N}^\delta = \delta\frac{\sym(\mathbf{I}^{\textup{T}}\mathbf{H}\mathbf{Y}\mathbf{V})}{\norm{\sym(\mathbf{I}^{\textup{T}}\mathbf{H}\mathbf{Y}\mathbf{V})}_{\mathcal{L}(\mathbb{R}^k)}}.
\end{equation*}
For each $C\in\mathcal{A}$, let the $(k-1)\times(k-1)$ matrices $\mathbf{B}$ and $\mathbf{C}$ be representations of $R(\gamma_0)$ and $R'(\gamma_0)\chi_C$ by $\mathbf{B}_{i,j} = I^{(j)} \cdot R(\gamma_0)I^{(i)}$ and $\mathbf{C}_{i,j} = I^{(j)}\cdot R'(\gamma_0)[\chi_C]I^{(i)}$. We use the definiteness of the following matrices to check the monotonicity relations (cf.~\eqref{eq:mono-ops})
\begin{align*}
\mathbf{A}_+ &= \mathbf{R}^\delta - \mathbf{B} - \beta\mathbf{C}, \\
\mathbf{A}_- &= \mathbf{B} - \betaU\beta\mathbf{C} - \mathbf{R}^\delta.
\end{align*}

\subsection{Numerical examples}

A standard finite element (FE) method with piecewise affine elements is applied to evaluate the involved PDE problems for the measurement maps. In dimension two the unit disk domain will be considered with $k = 32$ electrodes of size $\pi/k$ placed equidistantly on the boundary. In dimension three $k=64$ almost equidistant electrodes, given as spherical caps of radius $0.1$, will be placed on the unit sphere. For reconstruction we use a mesh with $8.0\times 10^4$ and $1.6\times 10^5$ nodes in two and three spatial dimensions, respectively. The mesh is more dense near the electrode positions to account for the weak singularities. For simulating the datum we use a different finer mesh which is also aligned with the inclusion boundaries for improved precision. The Fr\'echet derivative is evaluated on pixels and voxels with respective side length $\frac{1}{35}$ and $\frac{1}{20}$.

\begin{figure}[!htb]
	\centering
	\includegraphics[width = \textwidth]{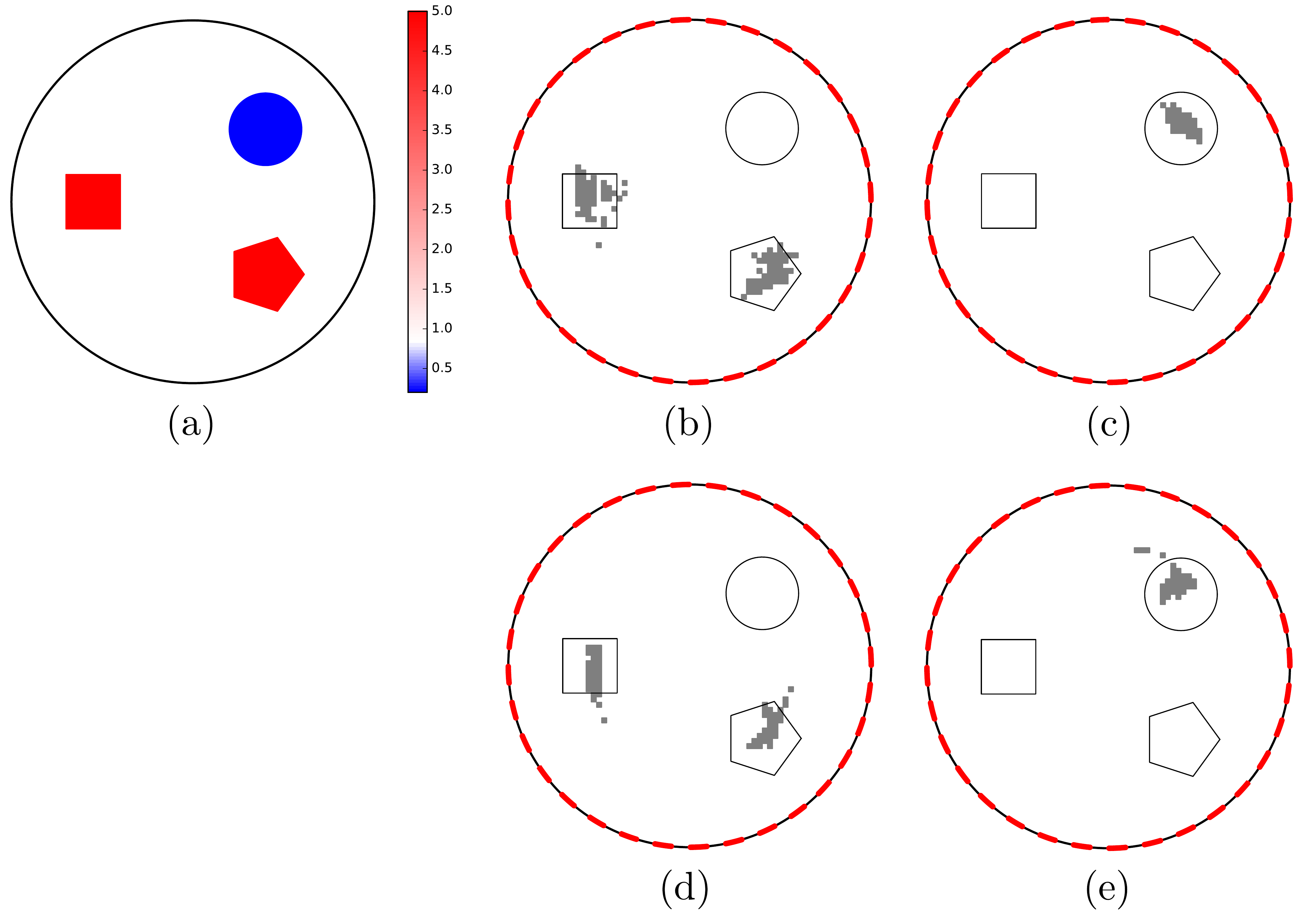}
	\caption{(a)~Two dimensional numerical phantom with positive part $\Dp$ (square and pentagon) and negative part $\Dm$ (ball). (b)~Reconstruction of $\Dp$ from noiseless datum with $\alpha_0 = 6.56\times 10^{-4}$. (c)~Reconstruction of $\Dm$ from noiseless datum with $\alpha_0 = 2.36\times 10^{-3}$. (d)~Reconstruction of $\Dp$ with $0.5\%$ noise and $\alpha_0 = 5.00\times 10^{-3}$. (e)~Reconstruction of $\Dm$ with $0.5\%$ noise and $\alpha_0 = 2.30\times 10^{-3}$.\label{fig:SPB}} 
\end{figure}

\begin{figure}[!htb]
	\centering
	\includegraphics[width = \textwidth]{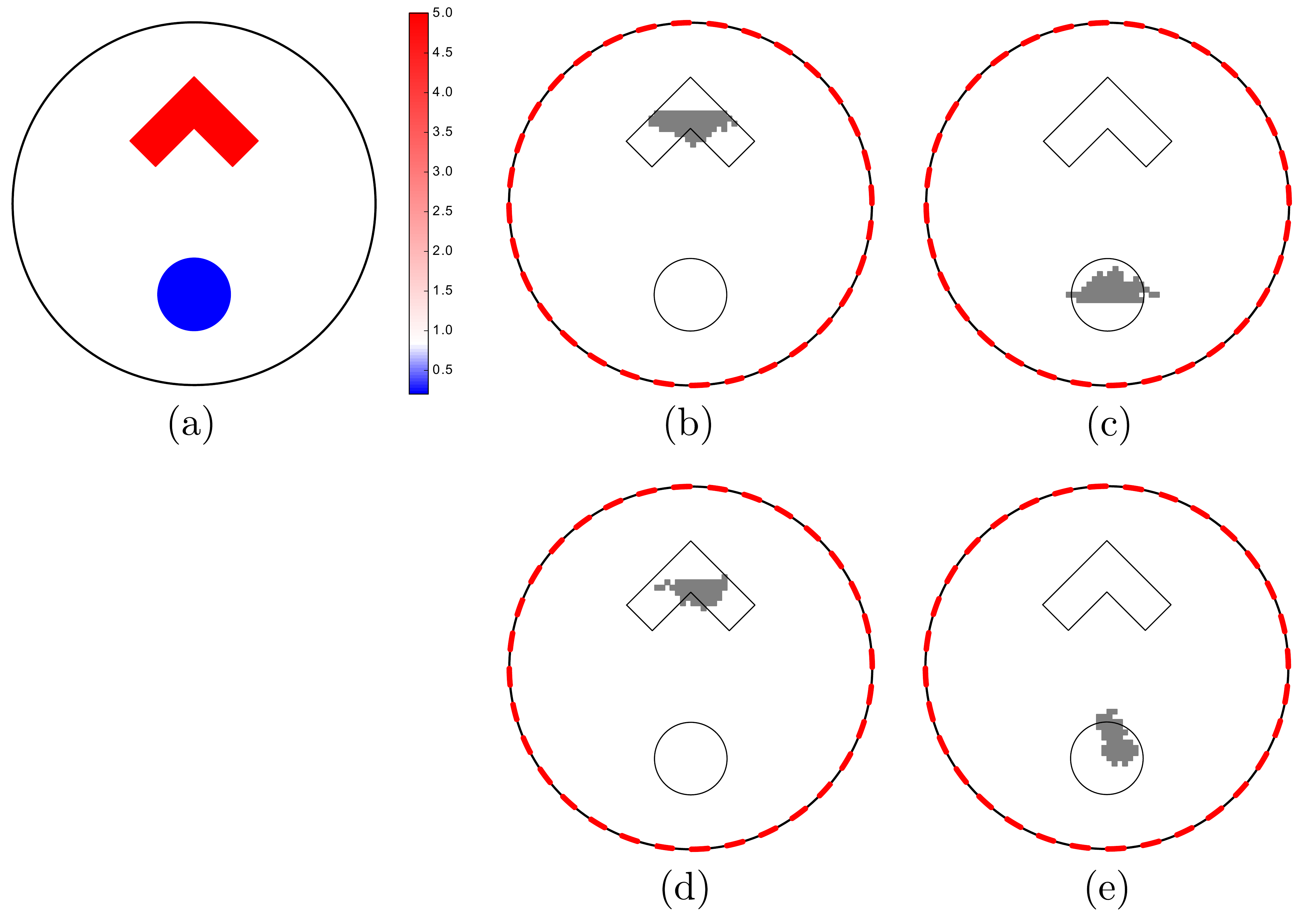}
	\caption{(a)~Two dimensional numerical phantom with positive part $\Dp$ (wedge) and negative part $\Dm$ (ball). (b)~Reconstruction of $\Dp$ from noiseless datum with $\alpha_0 = 9.00\times 10^{-4}$. (c)~Reconstruction of $\Dm$ from noiseless datum with $\alpha_0 = 6.72\times 10^{-4}$. (d)~Reconstruction of $\Dp$ with $0.5\%$ noise and $\alpha_0 = 4.20\times 10^{-3}$. (e)~Reconstruction of $\Dm$ with $0.5\%$ noise and $\alpha_0 = 3.26\times 10^{-3}$.\label{fig:VB}} 
\end{figure}

\begin{figure}[!htb]
	\centering
	\includegraphics[width = \textwidth]{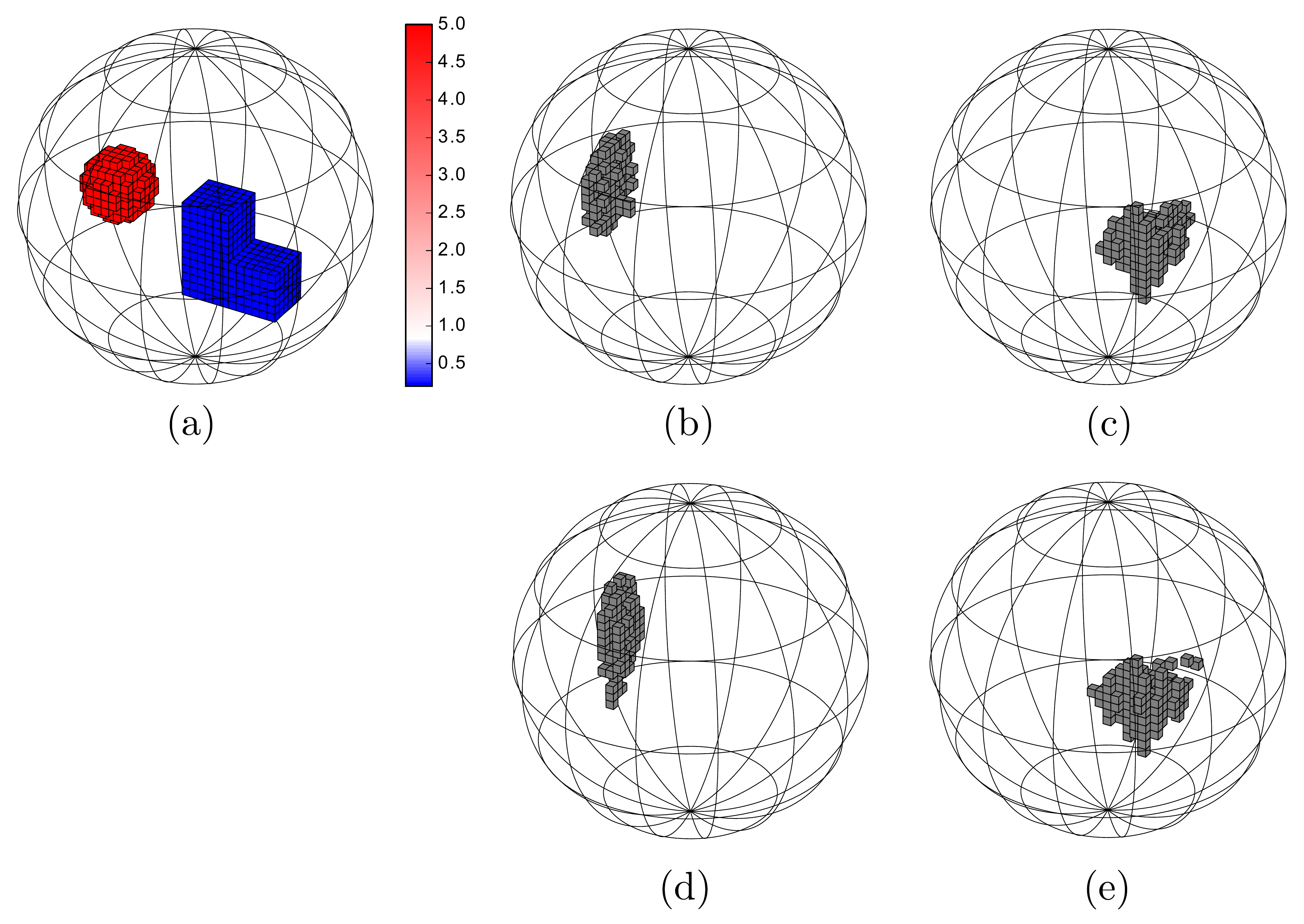}
	\caption{(a)~Three dimensional numerical phantom with positive part $\Dp$ (ball) and negative part $\Dm$ (L-shape). (b)~Reconstruction of $\Dp$ from noiseless datum with $\alpha_0 = 7.50\times 10^{-5}$. (c)~Reconstruction of $\Dm$ from noiseless datum with $\alpha_0 = 2.90\times 10^{-4}$. (d)~Reconstruction of $\Dp$ with $0.5\%$ noise and $\alpha_0 = 2.40\times 10^{-4}$. (e)~Reconstruction of $\Dm$ with $0.5\%$ noise and $\alpha_0 = 2.50\times 10^{-4}$.\label{fig:3D}} 
\end{figure}

\begin{figure}[!htb]
	\centering
	\includegraphics[width = .8\textwidth]{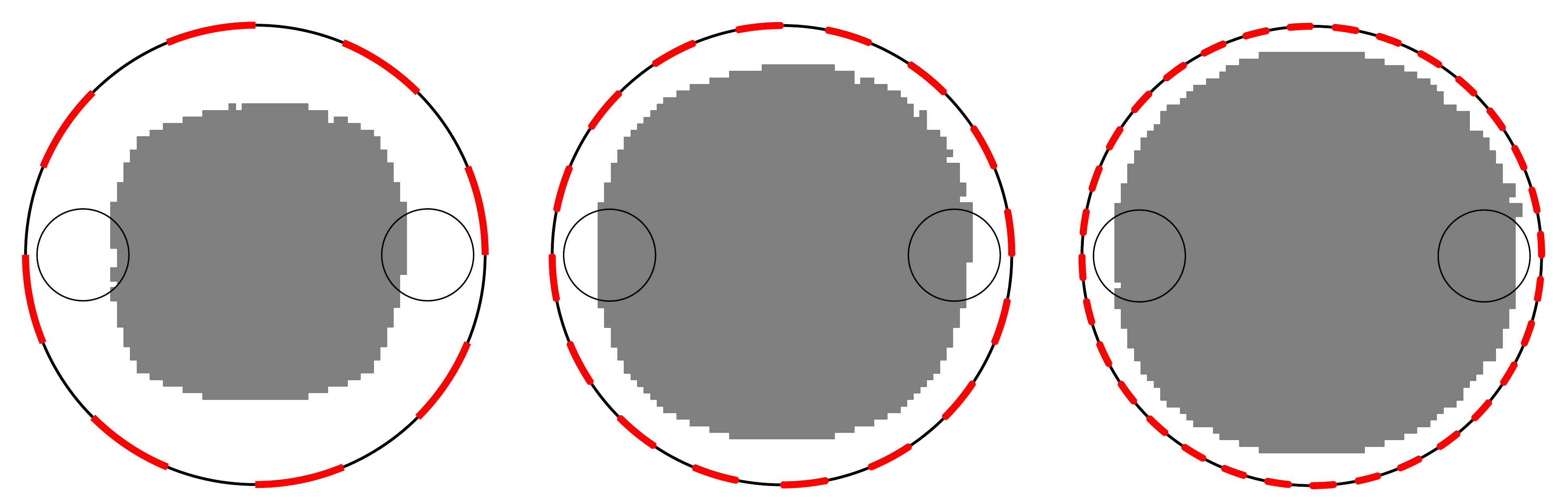}
	\caption{For $k = 8$, $16$, and $32$ electrodes of size $\pi/k$, upper bounds of reconstructions of $\Dp$ are shown, using regularization parameter $\alpha_0 = 0$. $\Dp$ is the ball outlined on the right side of the domain and $\Dm$ is on the left.\label{fig:elec_test}} 
\end{figure}

\begin{figure}[!htb]
	\centering
	\includegraphics[width = .6\textwidth]{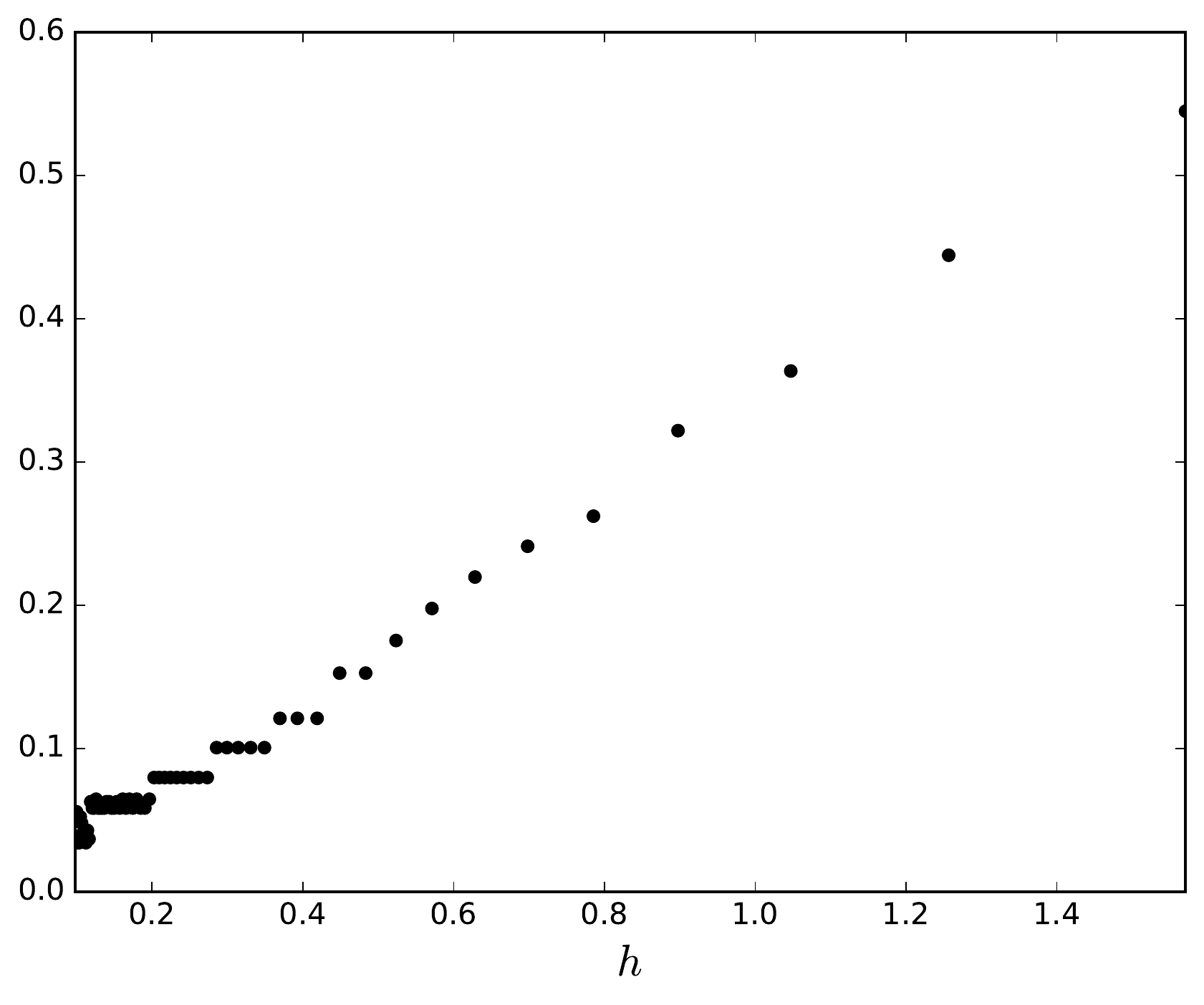}
	\caption{For $k = 4,5,\dots,64$ electrodes of size $\pi/k$, the distance $d_k$ from $\partial\Omega$ to upper bounds of reconstructions of $\Dp$ is plotted (cf.\ Figure~\ref{fig:elec_test}). $h = 2\pi/k$ is the corresponding maximal extended electrode diameter from \cite{Hyvonen09} and \cite[Theorems 2 and 3]{GardeStaboulis_2016}. \label{fig:dist}} 
\end{figure}

The conductivity $\gamma = 1 + 4\chi_{\Dp} - \frac{4}{5}\chi_{\Dm}$ will be used throughout, i.e.\ $\betaL = \betaU = 1$ and $\beta = 4$. The examples will include both reconstruction from a noiseless datum and from a noisy datum with $0.5\%$ relative noise.

The reconstructions can be seen in Figure~\ref{fig:SPB} and \ref{fig:VB} in the two dimensional case, and in Figure~\ref{fig:3D} for the three dimensional case. As expected from Theorem \ref{thm:conv}, the CEM reconstructions are generally smaller than the actual inclusions, however, it is observed that the locations of positive and negative inclusions can be independently reconstructed with clear separation, and that the algorithm performs robustly under noisy measurements. As is typically the case with reconstructions from the CEM, the convex inclusions are easier to reconstruct compared to the non-convex, as seen by the wedge-shape and the L-shape in Figure~\ref{fig:VB} and \ref{fig:3D}, though the general position of the inclusions is correctly reconstructed. On a laptop with an Intel i7 processor with CPU clock rate of 2.4 GHz, each reconstruction is computed in an average of 0.15 seconds in two dimensions and 3.3 seconds in three dimensions.

From the numerical experiments we noted that the reconstructions are not very sensitive to the choice of the regularization parameter, which can be changed up to almost an order of magnitude with only minor changes in the reconstruction. This is in stark contrast to the definite case which is very sensitive to the regularization parameter choice \cite{GardeStaboulis_2016}; this can be attributed the fact that test inclusions in the indefinite case are generally of significantly larger measure compared to the definite case, and the computations are therefore less sensitive to rounding errors. 

On the other hand, it turns out that reconstruction based on the CEM can be difficult near the measurement boundary, and in fact there can be invisible inclusions close to $\partial\Omega$. This can be illustrated by computing reconstructions of inclusions close to the boundary with regularization parameter $\alpha_0 = 0$, which will give an upper bound to any reconstruction with non-zero regularization. From Figure~\ref{fig:elec_test} these upper bounds are shown for $k = 8$, $16$, and $32$ electrodes, where it is observed that for $k = 8$ most of the inclusions will not be possible to reconstruct, and the distance $d_k$ of the upper bound to $\partial\Omega$ depends on the number of electrodes. This experiment is repeated for $k = 4,5,\dots,64$ equidistant electrodes of size $\pi/k$, where in the sense of the extended electrodes in \cite[Theorems 2 and 3]{GardeStaboulis_2016}, we have $h = 2\pi/k$. Figure \ref{fig:dist} suggests (up to the finite discretization) the conjecture that $d_k$ is $O(h)$, meaning the same rate of convergence as the CEM to the CM.  

Even though the indefinite method and definite method ultimately are based on the same type of monotonicity relations, the cause of the methods' different behaviour with approximate models is, to the the authors' knowledge, an open problem. 

\section{Conclusions}

We extended the regularization theory for the monotonicity method from the definite case to the indefinite case, and furthermore proved that positive and negative inclusions may be reconstructed separately by monotonicity-based reconstruction. For a regularization parameter choice criteria we proved that approximate forward models, including the CEM, can be used in the monotonicity method, and the involved monotonicity relations converge to the exact solution from the CM as the approximation error and noise level decay. 

We have presented a novel algorithm that implements the reconstruction method. Numerical examples, based on the CEM, provide evidence that the method is noise robust and capable of reconstructing inclusions with a realistic electrode model, provided that the inclusions are not too close to the measurement boundary.

\bibliographystyle{plain}
\bibliography{minbib}

\end{document}